\documentclass[10pt]{article}

\usepackage{lineno,hyperref}
\usepackage[left=1in, right=1in]{geometry}  
\modulolinenumbers[5]  

\usepackage{a4wide,graphicx}
\usepackage{amsmath,amssymb} 
\usepackage{amsthm} 
\usepackage{color}
\usepackage{caption}
\usepackage{comment}
\usepackage{enumerate}
\usepackage{bbm,yfonts,mathrsfs} 

\let\pa\partial  
  
\let\eps\varepsilon  
\let \wtilde \widetilde
\newcommand{\N}{{\mathbb N}} 
\newcommand{\R}{{\mathbb R}}

\newcommand{\dist}{\operatorname{dist}}

\newcommand{\dx}{\mathrm{d}x}

\newcommand{\K}{{\mathcal K}}
\newcommand{\A}{{\mathbb A}}
\newcommand{\B}{{\mathcal{B}}}
\newcommand{\dd}{{\mathrm{d}}}
\newcommand{\SO}{{\mathrm{SO}}}

\newcommand{\sym}{\operatorname{sym}}
\newcommand{\skw}{\operatorname{skw}}

\newcommand{\tensor}[1]{\mathbb #1}

\newcommand{\Khp}{\mathcal{K}_{(h)}^+}

\newcommand{\Khm}{\mathcal{K}_{(h)}^-}

\newcommand{\Kh}{\mathcal{K}_{(h)}}
\newcommand{\set}[1]{\mathcal #1}

\newcommand{\projyz}{\mathfrak{p}_{x'}}
\newcommand{\tw}{\mathfrak{t}}
\newcommand{\what}{\widehat}
\newcommand{\vphi}{\varphi}
\newcommand{\axl}{\operatorname{axl}}

\usepackage{mathtools} %Abs, norm et
\newcommand{\randomspace}{\Xi}
\newcommand{\randomelement}{\rho}
\newcommand{\trandomelement}{{\widetilde\rho}}
\newcommand{\randommeasure}{\mathbb P}
\newcommand{\drandommeasure}{\mathrm{d}\mathbb P(\randomelement)}
\newcommand{\randomalgebra}{\mathscr F}
\newcommand{\wtwoscale}{\xrightharpoonup{2}}
\newcommand{\stwoscale}{\xrightarrow{2}}
\newcommand{\weakly}{\rightharpoonup}

\newcommand{\stochsobolev}{\mathcal W}
\newcommand{\stochsmooth}{\mathcal C^\infty}
\newcommand{\twodomain}{\omega}

\DeclareMathOperator*{\esssup}{ess\,sup}
\DeclarePairedDelimiter\abs{\lvert}{\rvert}
\DeclarePairedDelimiter\norm{\lVert}{\rVert}
\DeclarePairedDelimiter\paren{(}{)}

\usepackage[normalem]{ulem}

\usepackage{color}
\usepackage[color]{changebar}

\cbcolor{blue}

\definecolor{green2}{RGB}{30,170,35}

\newtheorem{theorem}{Theorem}[section]   
\newtheorem{lemma}[theorem]{Lemma}   
\newtheorem{proposition}[theorem]{Proposition}    
\newtheorem{corollary}[theorem]{Corollary}  

\newtheorem{definition}{Definition}[section]

\newtheorem{remark}[theorem]{Remark}%[section]   

\begin{document}

%\begin{frontmatter}

\title{Derivation of homogenized Euler--Lagrange equations\newline for von K\'arm\'an rod}

\author{Mario Bukal\thanks{University of Zagreb, Faculty of Electrical Engineering 
and Computing, Unska 3, 10000 Zagreb, Croatia, \newline
E-Mails: mario.bukal@fer.hr, igor.velcic@fer.hr}, 
Matth\"aus Pawelczyk\thanks{Technische Universit\"at Dresden, Fachrichtung Mathematik, 
Institut f\"ur Geometrie, Zellerscher Weg 12--14, 01069 Dresden, Germany, 
E-Mail: matthaeus.pawelczyk@tu-dresden.de}, 
and Igor Vel\v ci\'c\footnotemark[1]}
%\ead{mario.bukal@fer.hr} 

%\author{Matth\"aus Pawelczyk}
% \ead{matthaeus.pawelczyk@tu-dresden.de} 
% 
% \author[addressFER]{Igor Vel\v ci\'c \corref{mycorrespondingauthor}}
% \cortext[mycorrespondingauthor]{Corresponding author}
% \ead{igor.velcic@fer.hr}
% 
%\address{University of Zagreb,}
% Faculty of Electrical Engineering and Computing, 
% Unska 3, 10000 Zagreb, Croatia}
% 
% \address[addressTUD]{Technische Universit\"at Dresden, Fachrichtung Mathematik, Institut f\"ur Geometrie, 
% Zellerscher Weg 12--14, 01069 Dresden, Germany}

\maketitle

\begin{abstract}
In this paper we study the effects of simultaneous homogenization and dimension reduction in the context of
convergence of stationary points for thin nonhomogeneous rods under the assumption
of the von K\'arm\'an scaling. Assuming stationarity condition for a sequence of deformations close to 
a rigid body motion, we prove that the corresponding sequences of scaled displacements and twist functions
converge to a limit point, which is the stationary point of the homogenized von K\'arm\'an rod model.
The analogous result holds true for the von K\'arm\'an plate model.
\end{abstract}

\vspace{4mm}

\vspace{2mm}
%\end{keyword}

%\end{frontmatter}

%\linenumbers

\section{Introduction} 
Boosted by the rigidity result of Friesecke, James and M\"uller \cite{FJM02}, rigorous derivation 
of various approximate models from three-dimensional nonlinear elasticity theory and its variational justification
has become a prominent research topic in the last decade. In particular, 
based on a refined rigidity result \cite{FJM06}, a whole hierarchy of limiting lower-domensional models 
has been derived by means of $\Gamma$-convergence techniques \cite{Bra02, DMa93}. For the context reasons, 
we only refer to the derivation of nonlinear inextensible rod models \cite{MoMu04, MoMu03}. In all these
models however, the material is assumed to be fixed, i.e.~non-oscillating. 
There is also a vast literature on studying the effects of simultaneous 
homogenization and dimesion reduction in various contexts \cite{BFF00, CoMo04, GuMo06}, 
but for the same reasons, we again focus on rod models. First 
attempts date back to \cite{JuTu89}, where the authors studied a linearized rod model assuming its homogeneity along 
the central line and nonhomogeneous microstructure in the cross section. A systemetic approach combining 
rigidity estimates \cite{FJM06} and the two-scale convergence method \cite{All92} 
was presented in \cite{Neu12} for the model of bending rod under the assumption of periodic structures. 
The same homogenized model has been obtained in
\cite{MaVe15} without periodicity assumptions, while using a $\Gamma$-convergence method tailored to the dimension
reduction in higher-order elasticity models. This method has been previously applied for the derivation of homogenized 
von K\'arm\'an plate \cite{Vel15} and linearized elasticity models \cite{BuVe15}, and in this paper we briefly outline how 
it accomplishes the homogenized von K\'arm\'an rod model (see Section \ref{sec:energy_rep}).  

The main purpose of this paper is to study convergence of stationary points of thin three-dimensional inhomogeneous 
rods in the von K\'arm\'an scaling regime. The above mentioned $\Gamma$-convergence techniques roughly assert that
a compact sequence of minimizers of scaled energies converges (on a subsequence) to a minimizer of the limit
energy. However, due to nonlinearities, these minimizers are typically only global and do not necessary 
satisfy the corresponding Euler--Lagrange equation. Convergence of stationary points of thin elastic rods in the
bending regime has been first studied in \cite{MMS07} on a simplified model of thin 2D strips 
and thenafter extended to the full 3D problem in \cite{MoMu08}. In order to identify the limit equations, 
besides the rigidity estimate, the authors also used compensated compactness and careful 
truncation arguments. Later on, convergence of stationary points of thin elastic rods in higher-order scaling regimes 
(including the von K\'arm\'an scaling) under physical growth conditions for the elastic energy density 
has been established in \cite{DaMo12}. However, in all these models the rod material was assumed to be fixed, i.e.~no oscillations 
of material are present.

In this paper we allow for possibility of oscillating (including random) materials and
study the effects of simultaneous homogenization and dimension reduction in the context of
convergence of stationary points in the von K\'arm\'an rod model. 
Let us denote by $\Omega = (0,L)\times\omega\subset\R^3$ a three-dimensional rod-like canonical domain 
of length $L>0$ and cross-section $\omega\subset \R^2$ having a Lipschitz boundary. The (scaled) energy functional of a rod of thickness $h>0$ 
occupying material domain $\Omega_h = (0,L)\times h\omega$ associated 
to a deformation $y^h : \Omega\to\R^3$ is defined on the canonical domain by
\begin{equation}\label{1.eq:energy}
\mathcal E^h(y^h) = \int_\Omega W^h(x,\nabla_hy^h)\dd x - \int_\Omega f^h\cdot y^h\dd x\,.
\end{equation} 
Above $W^h$ is an elastic energy density describing an addmissible composite material (see Section \ref{sec:gf}),  
$\nabla_hy^h = (\pa_1 y^h\, |\, \frac1h \pa_2y^h \, |\, \frac1h\pa_3y^h)$ 
denotes the scaled gradient of the deformation, and $f^h$ describes an external load. 
It is well known that different scaling regimes with respect to the thickness $h$ 
in the applied load and elastic energy 
lead at the limit to different rod models \cite{FJM06, Sca09}.
In the \emph{von K\'arm\'an scaling} of the rod, which is the subject of the research here,
we assume that the elastic energy of a sequence $(y^h)$ satisfies
\begin{equation}\label{1.eq:vKscaling}
\limsup_{h\downarrow0}\frac{1}{h^4}\int_\Omega W^h(x,\nabla_hy^h)\dd x < \infty\,.
\end{equation}
The forcing term scales as $f^h = h^3f$, where
$f=f_2e_2 + f_3e_3$ with $f_2,f_3\in L^2(0,L)$, meaning that
only normal loads to the mid-fiber of the rod are considered.
One can prove that under this scaling of the forces
the global minimizers satisfy the assumption (\ref{1.eq:vKscaling}).

Under assumption (\ref{1.eq:vKscaling})   on a sequence of deformations $(y^h)$ one can prove, 
based on the theorem of geometric rigidity \cite{FJM02},
that there exist sequences of rotations $(\bar R^h)\subset \SO(3)$ and constants $(c^h)\subset \R^3$,
such that transformed deformations $\hat{y}^h = (\bar{R}^h)^Ty^h - c^h$ converge to the identity deformation 
on $(0,L)$ in the $L^2$-norm, i.e.~$\hat y^h \to x_1e_1$, and moreover, 
$\nabla_h\hat y^h \to I$ in the $L^2$-norm \cite{MoMu04} (cf.~Theorem \ref{thm:rigidity} below).
Furthermore, scaled displacements, defined by 
\begin{align}\label{1.eq:uhvh}
u^h(x_1) &= \int_\omega\frac{\hat y_1^h - x_1}{h^2}\dd x'\,, \qquad
v_i^h(x_1) = \int_\omega \frac{\hat y_i^h}{h}\dd x'\,\ \text{ for }i=2,3\,,
\end{align}
and twist functions
\begin{equation}\label{1.eq:wh}
w^h(x_1) = \frac{1}{\mu(\omega)}\int_\omega\frac{x_2\hat y_3^h - x_3\hat y_2^h}{h^2}\dd x'\,,
\end{equation}
where
$\mu(\omega) = \int_\omega (x_2^2 + x_3^2)\dd x'$, converge (weakly) on a suitably extracted subsequence
to $(u,v_2,v_3,w)\in H^1(0,L)\times H^2(0,L)\times H^2(0,L)\times H^1(0,L)$ 
(see Theorem \ref{thm:comp2} for more details).

The strain sequence $(G^h)$ is implicitly defined through the decomposition of the scaled gradient as
$\nabla_h\hat y^h = R^h(I+h^2G^h)$, where $(R^h)$ denotes the sequence of rotation 
functions constructed in Theorem \ref{thm:rigidity}. Convergence results from Theorems \ref{thm:rigidity} and
\ref{thm:comp2} allow for the representation of the symmetrized strain $\sym G^h$ to the fixed and relaxation part
as follows:
\begin{equation}\label{1.eq:symGh}
\sym G^h = \sym (\imath(m)) + \sym\nabla_h\psi^h + o^h\,,
\end{equation}
where the fixed part comes from 
\begin{equation}\label{1.eq:mdef}
m = \left(\begin{array}{c}
	u' + \frac12\left((v_2')^2 + (v_3')^2\right) - v_2''x_2 - v_3''x_3\\
	-w'x_3 \\
	w'x_2
\end{array}
\right)\,,
\end{equation}
the relaxation sequence $(\psi^h)$ satisfies $(\psi_1^h,h\psi_2^h,h\psi_3^h)\to0$,
$\int_\omega(x_2\psi_3^h - x_3\psi_2^h)\dd x'\to 0$ in the $L^2$-norm and 
$\|\sym \nabla_h\psi^h\|_{L^2(\Omega)} \leq C$, while the rest sequence $(o^h)$ converges to zero in the $L^2$-norm. 
Utilizing the $\Gamma$-convergence method accomplished 
for the bending rod model in \cite{MaVe15}, we can analogously perform the simultaneous 
homogenization and dimension reduction process in the von K\'arm\'an case and obtain that the corresponding homogenized model, 
i.e.~the $\Gamma$-limit of $h^{-4}\mathcal E^h(\hat y^h)$ as $h\downarrow0$, is given by
\begin{equation}\label{1.eq:E0}
\mathcal E^0(u,v_2,v_3,w) = \Kh(m) - \int_0^L(f_2v_2 + f_3v_3)\dd x_1\,,
\end{equation}
where functions $u$, $v_2$, $v_3$ and $w$ are the weak limits of scaled displacements and twist function, respectively, 
and $m$ is given by (\ref{1.eq:mdef}). 
Moreover,
the resulting limit elastic energy density (depending on a given subsequence of the diminishing thickness $(h)$) 
can be calculated according to
\begin{equation}\label{1.eq:Khrep}
\Kh(m) = \lim_{h\downarrow0}\int_{\Omega} Q^{h}(x,\imath(m)  + \sym\nabla_{h}\psi^{h}_m)\dd x\,,
\end{equation}
where $Q^h$ is the quadratic form approximating the energy density $W^h$, and $(\psi^h_m)$ the
corresponding relaxation sequence. Confer Section \ref{sec:energy_rep} for more details.

As we already stressed out, our aim is to study the stationary points of the energy functional $ \mathcal E^h $ rather than just 
global minimizers attainable through the $\Gamma$-convergence techniques. The weak form of the Euler--Lagrange equation of the
functional $\mathcal E^h$, assuming the zero boundary condition
on the zero cross-section $\{0\}\times \omega$, formally reads:
\begin{equation}\label{int.eq:EL1}
\int_\Omega\left(DW^h(x,\nabla_h y^h):\nabla_h\phi - h^3(f_2\phi_2 + f_3\phi_3)\right)\dd x = 0\,,
\end{equation}
for all test functions 
$\phi\in H_\omega^1(\Omega,\R^3) = \{\phi\in H^1(\Omega) \ :\ \phi|_{\{0\}\times\omega} = 0\}$. 
This notion of stationarity is the standard one, but possibly not best suited for the nonlinear elasticity. 
Namely, it is still an open question whether under physical growth assumptions on the energy densities $W^h$, 
global or suitably defined local minimizers of $\mathcal E^h$ satisfy the Euler--Lagrange equation \cite{Ball02}. 
To prevail this issue here we use unphysical assumptions on a linear growth and continuity 
of the stress (cf.~hypothesis H5 below). 
There is an alternative notion of first-order stationarity in elasticity, proposed by Ball in \cite{Ball02}, 
and that concept is compatible with a physical growth condition
which roughly says that the energy blows up if the deformation degenerates. 
While the authors in \cite{DaMo12} 
managed to deal with the alternative stationarity condition and to systematically derive the 
corresponding stationarity conditions for the limit models, our method is not compatible with that mainly because of the possibility of interpentration of the matter and we reside in this setting.

Now we are in position to state the main result of the paper.
\begin{theorem}\label{thm:main}
Let the sequence $(W^h)$ describes an admissible composite material and let $({y}^h)$ 
be a sequence satisfying (\ref{1.eq:vKscaling}). Then the transformed deformations $(\hat{y}^h)$
and sequences of scaled displacements converge (on a subsequence) as follows:
% \begin{equation*}
% \hat y^h\to x_1e_1\quad\text{in }H^1(\Omega,\R^3)\,.
% \end{equation*}
% Furthermore, 
\begin{align*}
\hat y^h&\to x_1e_1\quad\text{in }H^1(\Omega,\R^3)\,,\\
u^h &\rightharpoonup u \text{ weakly in }H^1(0,L)\,,\\
v_i^h &\to v_i \text{ strongly in }H^1(0,L)\,,\text{ and } v_i\in H^2(0,L)\text{ for } i=2,3\,,\\
w^h &\rightharpoonup w \text{ weakly in }H^1(0,L)\,.
\end{align*}
Let $f^h = h^3(f_2e_2 + f_3e_3)$ with $f_2,f_3\in L^2(0,L)$ be an external load and
assume that $(\hat{y}^h)$ are stationary points of 
the energy functional $\mathcal E^h$, i.e.~solve equation (\ref{int.eq:EL1}), 
then $(u,v_2,v_3,w)$ is a stationary point of 
the limit energy functional $\mathcal E^0$.
\end{theorem}
Big part of the proof of Theorem \ref{thm:main} (compactness) does not differ much from the non-oscillating case, 
which is already available in the literature. These results are comprised and properly referenced in Theorems \ref{thm:rigidity} and \ref{thm:comp2} below
in Section \ref{sec:prel}.
Hence, the main focus here is on the statement that stationarity of the transformed deformations $\hat y^h$ of the energy
functional $\mathcal E^h$ (in the sense of (\ref{int.eq:EL1})) implies the stationarity of the point
$(u,v_2,v_3,w)$ for the limit energy functional $\mathcal E^0$. The key point in proving that statement is the 
orthogonality property provided in Lemma \ref{thm:orth_prop}, which essentially allows us to identify
two relaxation sequences: $(\psi^h)$ from (\ref{1.eq:symGh}) and $(\psi^h_m)$ from (\ref{1.eq:Khrep}),
up to $L^2$-concentrations, which are irrelevant for identification of weak limits. 
The proof of Lemma \ref{thm:orth_prop}, together with the proof of Theorem \ref{thm:main}, and identification
of limit Euler--Lagrange equations are devised to Section \ref{sec:homEL}, while some technical results can be found
in the appendix. We emphasize at this point that, up to some technical peculiarities, the same approach can be utilized for studying the 
convergence of stationary points of the von K\'arm\'an plate model, and the analogous result holds true.

Finally, in Section \ref{sec:sthom} we consider randomly oscillating materials satisfying the von K\'arm\'an scaling and 
provide an explicit cell formula for the limit energy density of the functional $\Kh$ (cf.~Proposition \ref{prop:intrep}). 
This result also covers the case of periodically and almost periodically oscillating materials.

%%%%%%%%%%%%%%%%%%%%%%%%%%%%%%%%%%%%%%%%%%%%%%%%%%%%%

\section{Preliminaries}\label{sec:prel}

\subsection{Notation}
$\Omega= (0,L)\times\omega\subset\R^3$ is a Lipschitz domain describing the canonical configuration of a rod of
length $L>0$ and shape $\omega\subset\R^2$. Vectors $e_1,e_2,e_3$ denote the canonical basis of $\R^3$ and 
$(x_1,x')\in\R^3$, with $x'=(x_2,x_3)\in\R^2$, denote the coordinates of a 
point in $\R^3$ with respect to that basis. Also, we will frequently use the projection of a point $x\in\R^3$
to $x'$-plane, denoted by $\projyz(x) = (0,x')^T$. 
For a given thickness $h>0$, the scaled gradient is denoted by 
$\nabla_h = (\pa_1,\frac{1}{h}\pa_2,\frac{1}{h}\pa_3)$. The space of real $3\times3$ matrices is denoted by 
$\R^{3\times3}$, while $\R^{3\times3}_{\sym}$, $\R^{3\times3}_{\skw}$ and $\SO(3)$ denote the subspaces
of symmetric, skew-symmetric, and special orthogonal matrices, respectively. 
For a skew-symmetric matrix $A$ we denote its axial vector by $\axl A = (A_{32},A_{13},A_{21})$.
By $\iota : \R^3 \to \R^{3\times3}$ we denote the inclusion $\iota(v) = v \otimes e_1$. 
Depending on the context, by $|\cdot|$ we
denote both the Lebesgue measure of a set and the euclidean norm of a vector in $\R^d$. 
The space of smooth functions on $(0,L)$ which are vanishing at zero will be denoted by $C^\infty_0(0,L)$.
Given two functions
$\phi,\psi\in L^1(\Omega,\R^3)$, we define the {\em twist} function $\tw(\phi,\psi):(0,L)\to\R$ by
\begin{equation*}
\tw(\phi,\psi)(x_1) = \int_\omega(x_2\psi - x_3\phi)\dd x'\,.
\end{equation*}
Finally, the moments of a function $\Psi\in L^1(\Omega,\R^{3\times3})$ are denoted as follows. The zeroth moment
$\overline \Psi : (0,L)\to \R^{3\times3}$ is defined by
\begin{equation}\label{1.def:moment0}
\overline\Psi(x_1) = \int_\omega\Psi(x)\dd x'\,,
\end{equation}
and first-order moments $\wtilde \Psi,\what\Psi : (0,L)\to \R^{3\times3}$ are defined by
\begin{equation}\label{1.def:moments1}
\wtilde\Psi(x_1) = \int_\omega x_2\Psi(x)\dd x'\,,\qquad \what\Psi(x_1) = \int_\omega x_3\Psi(x)\dd x'\,.
\end{equation}

\subsection{von K\'arm\'an rod model -- supplement}\label{sec:gf}
Let $\omega\subset \R^2$ be a Lipschitz domain of the Lebesgue measure $|\omega|=1$ and assume that
coordinate axes are chosen such that
\begin{equation*}
\int_\omega x_2 \dd x' = \int_\omega x_3\dd x' = \int_\omega x_2x_3 \dd x' = 0\,.
\end{equation*}   
By $\Omega^h = (0,L)\times h\omega$ we denote the reference configuration (material domain) 
of a rod-like body of thickness $h >0$ and length $L>0$.
Performing the standard change of variables $\Omega_h \ni \hat{x}\mapsto x\in\Omega$, given by
$x_1 = \hat{x}_1$, $x' = \frac1h\hat x'$, we will in the sequel work on the canonical domain 
$\Omega = (0,L)\times\omega$.
For every $h>0$, the (scaled) energy functional of a deformation $y^h : \Omega\to\R^3$ is given by 
expression (\ref{1.eq:energy}).

For the elastic energy densities $W^h$ we have more or less standard 
hypotheses on nonlinear composite material, which are listed in the sequel.

{\em \noindent Nonlinear material law.} 
Let $\alpha$, $\beta$, $\varrho$ and $\kappa$ be positive constants such that $\alpha \leq \beta$. 
The class $\mathcal{W}(\alpha,\beta,\varrho,\kappa)$ consists of all measurable functions 
$W:\R^{3\times3}\to [0,+\infty]$ satisfying:
\begin{description}
\item[(H1)] frame indifference: $W(RF) = W(F)$ for all $F\in\R^{3\times3}$ and $R\in\SO(3)$;
\item[(H2)] non-degeneracy: 
\begin{align*}
W(F) & \geq \alpha\dist^2(F,\SO(3))\quad\text{for all } F\in\R^{3\times3}\,,\\
W(F) & \leq \beta\dist^2(F,\SO(3))\quad\text{for all } F\in\R^{3\times3}\text{ with }\dist^2(F,\SO(3))\leq\varrho\,;
\end{align*}
\item[(H3)] minimality at identity: $W(I) = 0$;
\item[(H4)] quadratic expansion at identity: $W(I + G) = Q(G) + o(|G|^2)$ as $G\to0$ ($G\in\R^{3\times3}$), 
where $Q:\R^{3\times3}\to \R$ is a quadratic form;
\item[(H5)] linear stress growth: $|DW(F)| \leq \kappa(|F| + 1)$ for all $F\in\R^{3\times3}$. 
\end{description}
{\em\noindent Admissible composite material.}
For $\alpha$, $\beta$, $\varrho$ and $\kappa$ positive constants as above, a family of functions
$W^h :\Omega\times\R^{3\times3}\to [0,+\infty]$ describes an admissible composite material of class 
$\mathcal{W}(\alpha,\beta,\varrho,\kappa)$ if the following hypotheses hold:
\begin{description}
\item[(C1)] for every $h>0$, $W^h$ is almost everywhere equal to a Borel function on $\Omega\times\R^{3\times3}$;
\item[(C2)] for every $h>0$, $W^h(x,\cdot)\in \mathcal{W}(\alpha,\beta,\varrho,\kappa)$ for a.e.~$x\in\Omega$;
\item[(C3)] there exists a monotone function $r : \R_+ \to (0,+\infty)$ such that $r(\delta)\downarrow 0$ 
as $\delta\downarrow0$ and 
\begin{equation}\label{eq:ass}
\forall G\in\R^{3\times3}\,, \ \forall h>0\ :\ \esssup_{x\in\Omega}|W^h(x,I+G) - Q^h(x,G)| \leq r(|G|)|G|^2\,,
\end{equation}
where $Q^h(x,\cdot)$ are quadratic forms defined in (H4).
\end{description}
The given quadratic form $Q^h(x,\cdot)$ can be (uniquely) represented 
by a positive semidefinite linear operator $\A^h(x)$,
\begin{equation*}
Q^h(x, F) = \frac12\A^h(x)F:F\,,\quad \text{for all } F\in \R^{3\times3}\text{ and for a.e.~}x\in\Omega\,.
\end{equation*}
Assuming that $Q^h$ corresponds to an elastic energy density $W^h$
belonging to a family of elastic energy densities describing an admissible composite material of 
class $\mathcal{W}(\alpha,\beta,\varrho,\kappa)$, one can easily prove:
\begin{enumerate}[(a)]
  \item $\alpha|\sym F|^2 \leq Q^h(x,F) = Q^h(x,\sym F) \leq \beta|\sym F|^2$, for all $F\in\R^{3\times3}$;
  \item $|Q^h(x,F_1) - Q^h(x,F_2)|\leq \beta |\sym F_1 - \sym F_2||\sym F_1 + \sym F_2|$, 
  for all $F_1, F_2 \in\R^{3\times3}$.
\end{enumerate}

\subsection{Rigidity and compactness}
Using the theorem of geometric rigidity \cite{FJM02}, the following result has been established in
\cite{MoMu04}.   
\begin{theorem}\label{thm:rigidity}
Let $(y^h)\subset H^1(\Omega,\R^3)$ be a sequence satisfying
\begin{equation*}
\limsup_{h\downarrow0}\frac{1}{h^4}\int_\Omega\dist^2(\nabla_h y^h,\SO(3))\dd x < +\infty\,.
\end{equation*} 
Then there exist: a sequence of maps $(R^h)\subset C^\infty([0,L],\SO(3))$, a sequence of constant 
rotations $(\bar R^h)\subset \SO(3)$ and constants $(c^h)\subset\R^3$ such that the sequence of
deviations from the rigid motion $(\hat y^h)$, defined by $\hat{y}^h = (\bar{R}^h)^Ty^h - c^h$, satisfies
\begin{align}
\|\nabla_h\hat y^h - R^h\|_{L^2(\Omega)} &\leq Ch^2\,,\label{eq:rige1}\\
\|(R^h)'\|_{L^2(0,L)} &\leq Ch\,,\\
\|R^h - I\|_{L^2(0,L)} &\leq Ch\,.
\end{align}
\end{theorem}
\noindent The sequence of constants $(c^h)$ in the previous theorem can be chosen such that 
\begin{equation*}
\int_\Omega(\hat{y}_1^h - x_1)\dd x = 0\,,\qquad \int_\Omega\hat{y}^h_i\dd x = 0\,
\quad\text{for }i=2,3\,.
\end{equation*}
Next, we introduce the following ansatz for $(\hat y^h)$: 
\begin{equation}\label{def:ansatz}
\arraycolsep=1.4pt\def\arraystretch{1.7}
\begin{array}{l}\displaystyle
\hat{y}_1^h = x_1 + h^2\left(u^h - x_2\frac{R^h_{21}}{h} - x_3\frac{R^h_{31}}{h}\right) 
+ h^2\beta_1^h\,,\\
\hat{y}_i^h = hx_i + hv_i^h + h^2w^h x_i^\bot
+ h^2\beta_i^h\,,\ \text{ for }i=2,3\,,
\end{array}
\end{equation}
where $x^\bot = (0,-x_3,x_2)$, and functions $u^h$, $v_2^h$, $v_3^h$, and $w^h$ are defined in (\ref{1.eq:uhvh}) and
(\ref{1.eq:wh}).
\begin{remark}
Observe that the proposed ansatz is a slight modification of the ansatz for the same sequence $(\hat y^h)$
from \cite[Theorem 2.2 (f)]{MoMu04}.
In lieu of terms $(v_i^h)'$, $i=2,3$, we set $\frac{1}{h}R^h_{i1}$, respectively. This enables us to 
control the full scaled gradient of the corrector sequence $(\beta^h)$ in the $L^2$-norm 
(see Theorem \ref{thm:comp2} below), which is crucial for application of our method in the analysis afterwards.
\end{remark}
\begin{theorem}\label{thm:comp2} Let the assumption and notation of the previous theorem be retained.
For sequences $(u^h)$, $(v_i^h)$, $i=2,3$, and $(w^h)$ defined above, we have the following convergence results 
which hold on a subsequence:
\begin{align*}
u^h &\rightharpoonup u \text{ weakly in }H^1(0,L)\,;\\
v_i^h &\to v_i \text{ strongly in }H^1(0,L)\,,\text{ and } v_i\in H^2(0,L)\text{ for } i=2,3\,;\\
w^h &\rightharpoonup w \text{ weakly in }H^1(0,L)\,.
\end{align*}
Moreover, the sequence of corrector functions $(\beta^h)$ satisfies uniform bounds: 
$\|\beta^h\|_{L^2(\Omega)}\leq Ch$ and 
$\|\nabla_h\beta^h\|_{L^2(\Omega)} \leq C$.
\end{theorem}
\begin{proof}The proof follows the lines of the proof of Theorem 2.2 from \cite{MoMu04}, but we include it here
for the reader's convenience.
Let us define
\begin{align*}
A^h := \frac{1}{h}(R^h - I)\,.
\end{align*}
From the previous theorem we have $\|R^h - I\|_{L^2(0,L)} \leq Ch$ and 
$\|(R^h)'\|_{L^2(0,L)} \leq Ch$, which implies the uniform bound $\|A^h\|_{H^1(0,L)}\leq C$. 
Therefore, (up to a subsequence)
$A^h \rightharpoonup A$ weakly in $H^1((0,L),\R^{3\times3})$. From the compactness of 
the Sobolev embedding $H^1((0,L),\R^{3\times3})\hookrightarrow L^\infty((0,L),\R^{3\times3})$,
we conclude $A^h \to A$ strongly in $L^\infty((0,L),\R^{3\times3})$. Direct calculation reveals
the identities
\begin{equation*}
A^h + (A^h)^T = -hA^h(A^h)^T\quad\text{and}\quad \frac{1}{h^2}\sym(R^h - I) = \frac{1}{2h}(A^h + (A^h)^T)\,,
\end{equation*}
which respectively imply $A^T = -A$ and 
\begin{equation}\label{eq.conv.sym}
\frac{1}{h^2}\sym(R^h - I) \to \frac{1}{2}A^2\quad\text{strongly in }L^\infty((0,L),\R^{3\times3})\,.
\end{equation}
Since $\|\nabla_h\hat y^h - R^h\|_{L^2(\Omega)} \leq Ch^2$, using the triangle inequality and established 
convergence results, we conclude
\begin{equation}\label{eq:A_strong}
\frac{1}{h}(\nabla_h\hat y^h - I) \to A \quad\text{strongly in } L^2(\Omega,\R^{3\times3})\,.
\end{equation}
By the construction $\int_0^Lu^h(x_1)\dd x_1 = 0$. Thus, the Poincar\'e and Jensen inequalities together 
with (\ref{eq.conv.sym}) imply
\begin{align*}
\|u^h\|_{L^2(0,L)} & \leq C_P\|(u^h)'\|_{L^2(0,L)}\leq \frac{C_P}{h^2}\|\pa_1\hat y_1^h - 1\|_{L^2(\Omega)} \\
& \leq \frac{C_P}{h^2}\|\pa_1\hat y_1^h - R^h_{11}\|_{L^2} + \frac{C_P}{h^2}\|R^h_{11} - 1\|_{L^2}\leq C\,.
\end{align*}
Therefore, up to a subsequence $u^h \rightharpoonup u$ weakly in $H^1(0,L)$.
Similarly, $\int_0^Lv_i^h(x_1)\dd x_1 = 0$ for $i=2,3$, and 
\begin{equation*}
\|(v_i^h)'\|_{L^2(0,L)}\leq \frac{1}{h}\|\pa_1\hat y_i^h\|_{L^2(\Omega)}\leq C\,.
\end{equation*}
Hence, (up to a subsequence) $v_i^h \rightharpoonup v_i$ weakly in $H^1(0,L)$. Moreover, since
\begin{equation*}
(v_i^h)' = \int_\omega\frac{\pa_1 \hat y_i^h}{h}\dd x' \to A_{i1}
\quad\text{strongly in } L^2(\Omega,\R^{3\times3})\,,
\end{equation*}
one concludes that $A_{i1} = v_i'$ for $i=2,3$. Since $A_{i1}\in H^1(0,L)$, we conclude $v_i\in H^2(0,L)$ 
for $i=2,3$.
Next, we consider the sequence of twist functions $(w^h)$. Note that they can be written as
\begin{align*}
w^h(x_1) &= \frac{1}{\mu(\omega)}\int_\omega x_2\left(\frac{h^{-1}\hat y_3^h - x_3}{h} - 
\frac{1}{h^2} \int_\omega \hat y_3^h \dd x'\right)\dd x'\\
&\qquad - \frac{1}{\mu(\omega)}\int_\omega x_3\left(\frac{h^{-1}\hat y_2^h - x_2}{h} - 
\frac{1}{h^2} \int_\omega \hat y_2^h \dd x'\right)\dd x'\,.
\end{align*}
For the above integrands we have (according to \eqref{eq:A_strong} and the Poincar\'e inequality):
\begin{align*}
\frac{h^{-1}\hat y_3^h - x_3}{h} - \frac{1}{h^2} \int_\omega \hat y_3^h \dd x' &\to A_{32}x_2
\quad\text{strongly in }L^2(\Omega)\,;\\
\frac{h^{-1}\hat y_2^h - x_2}{h} - \frac{1}{h^2} \int_\omega \hat y_2^h \dd x' &\to -A_{32}x_3
\quad\text{strongly in }L^2(\Omega)\,.
\end{align*}
Therefore, $w^h$ converges strongly in the $L^2$-norm to the function $w = A_{32} \in L^2(0,L)$.
Using the a priori estimate $\|\nabla_h\hat y^h - R^h\|_{L^2} \leq Ch^2$ and the 
normality of rotation matrix columns, we conclude the uniform bound $\|(w^h)'\|_{L^2(0,L)} \leq C$. Hence,
$w^h\rightharpoonup w$ weakly in the $H^1$-norm. Observe that the limit matrix 
$A\in H^1((0,L),\R^{3\times3})$ is completely identified
by the limits $u, w\in H^1(0,L)$ and $v_1, v_2\in H^2(0,L)$ in the following way
\begin{equation}\label{eq:A}
 A = \left(
 		\begin{array}{ccc}
 				0    & -v_2' & -v_3' \\
 				v_2' & 0  & -w \\
 				v_3' & w  & 0
 		\end{array}
 	 \right)\,.
\end{equation}

Finally, we consider the sequence of corrector functions $(\beta^h)$ given by:
\begin{align*}
\beta_1^h(x) &= \frac{\hat y_1^h(x) - x_1}{h^2} - u^h(x_1) + x_2\frac{R^h_{21}(x_1)}{h} + 
x_3\frac{R^h_{31}(x_1)}{h} \,;\\
\beta_i^h(x) &= \frac{1}{h}\left(\frac{\hat y_i^h(x) - hx_i}{h} - v_i^h(x_1) - hw^h(x_1)x_i^\bot\right)\,,
\quad i=2,3\,.
\end{align*}
For brevity reasons, let us denote $\pa^h_i = \frac{1}{h}\pa_i$, then for $i=2,3$ we compute 
\begin{align*}
\pa_i\beta_1^h = \frac{1}{h^2}\pa_i\hat{y}_1^h + \frac{R_{i1}^h}{h} = 
\frac{1}{h}\left(\pa^h_i\hat{y}_1^h - R_{1i}^h\right) + \frac{R_{i1}^h + R_{1i}^h}{h}\,.
\end{align*}
The first term on the right-hand side is bounded in the $L^2$-norm due to $\|\nabla_h\hat y^h - R^h\|_{L^2(\Omega)} \leq Ch^2$,
and the second one due to (\ref{eq.conv.sym}). Thus, $\|\pa_i\beta_1^h\|_{L^2(\Omega)}\leq Ch$ for $i=2,3$.
Since $\int_\omega\beta_1^h(x)\dd x' = 0$, using the Poincar\'e inequality we conclude 
\begin{equation*}
\|\beta_1^h(x_1,\cdot)\|_{L^2(\omega)}^2 \leq C(\|\pa_2\beta_1^h(x_1,\cdot)\|_{L^2(\omega)}^2 
+ \|\pa_3\beta_1^h(x_1,\cdot)\|_{L^2(\omega)}^2)\,.
\end{equation*}
Integrating the latter inequality along $x_1$-direction yields the $L^2(\Omega)$-bound 
on $\beta_1^h$ of order $O(h)$.
The identity
\begin{equation*}
\pa_1\beta_1^h = \frac{\pa_1\hat y_1^h - 1}{h^2} - (u^h)' + x_2\frac{(R^h_{21})'}{h} + 
x_3\frac{(R^h_{31})'}{h}\,, 
\end{equation*}
directly implies the uniform bound $\|\pa_1\beta_1^h\|_{L^2(\Omega)}\leq C$.
Straightforward calculations reveal 
\begin{align*}
\pa_j\beta_i^h = \frac{1}{h}\left(\pa_j^h\hat y_i^h - \delta_{ij} - (-1)^j(1-\delta_{ij})hw^h\right)\,,
\quad\text{for }i,j=2,3\,,
\end{align*}
where we have used $\pa_jx_i^\bot = (-1)^j(1-\delta_{ij})$. Furthermore,
\begin{align*}
(\sym\nabla\beta^h)_{ij} = \frac{\pa_j\beta_i^h + \pa_i\beta_j^h}{2} = 
\frac{1}{h}\left(\sym(\nabla_h\hat{y}^h - I)\right)_{ij}\,,\quad\text{for }i,j=2,3\,,
\end{align*}
which implies the uniform bound $\|(\sym\nabla\beta^h)_{ij}\|_{L^2(\Omega)} \leq Ch$ for $i,j=2,3$.
Note that for a.e.~$x_1\in (0,L)$ the function $(\beta_2^h(x_1,\cdot),\beta_3^h(x_1,\cdot))$ belongs to the 
closed subspace
\begin{equation*}
\B = \left\{\alpha\in H^1(\omega,\R^2)\ :\ \int_\omega\alpha(x')\dd x' = 0\,, 
\int_\omega(x_3\alpha_2 - x_2\alpha_3)\dd x' = 0\right\}\,,
\end{equation*}
on which a Korn type inequality \cite{OSY92} holds 
\begin{equation*}
\|\beta_2^h(x_1,\cdot)\|_{H^1(\omega)}^2 + \|\beta_3^h(x_1,\cdot)\|_{H^1(\omega)}^2 
\leq C\sum_{i,j=2,3}\|(\sym\nabla\beta^h(x_1,\cdot))_{ij}\|_{L^2(\omega)}^2\,.
\end{equation*}
Integrating the latter with respect to $x_1$, yields the respective uniform $H^1(\Omega)$-bound.
Hence, we proved $\|\beta^h\|_{L^2(\Omega)}\leq Ch$.
Finally,
\begin{align*}
\pa_1\beta_i^h &= \frac{1}{h}\left(\frac{\pa_1\hat y_i^h}{h} - (v_i^h)' - h(w^h)' x_i^\bot\right)\\
& = \frac{1}{h}\left(\frac{\pa_1\hat y_i^h - R^h_{1i}}{h} - 
\frac{1}{h}\int_\omega(\pa_1\hat y_i^h - R^h_{1i})\dd x'
 - h(w^h)' x_i^\bot\right)
\,,\quad\text{for }i,j=2,3\,,
\end{align*}
and the previously established convergence results imply $\|\pa_1\beta_i^h\|_{L^2(\Omega)}\leq C$.
Thus, we have proved $\|\nabla_h\beta^h\|_{L^2(\Omega)} \leq C$.
\end{proof}

\subsection{Strain and stress estimates}
For every $h>0$, using the rotation matrix function $R^h$, the strain tensor $G^h$ is implicitly defined 
through the following decomposition of the scaled deformation gradient
\begin{equation*}
\nabla_h\hat y^h = R^h(I + h^2G^h)\,.
\end{equation*} 
The explicit identity $G^h = h^{-2}(R^h)^T(\nabla_h\hat y^h - R^h)$ directly implies with (\ref{eq:rige1}) the $L^2$-uniform bound
on the sequence $(G^h)$. Hence, there exists $G\in L^2(\Omega,\R^{3\times3})$ such that $G^h\rightharpoonup G$ on a
subsequence. Our aim is to describe the symmetrized strain $\sym G^h$ in more detail. 
First, we explicitly involve the limit functions $u, w\in H^1(0,L)$ and $v_1, v_2\in H^2(0,L)$ 
into our ansatz (\ref{def:ansatz}) in the following way:
\begin{align*}
\frac{\hat y_1^h - x_1}{h^2} &= u - x_2v_2' - x_3v_3' + \psi_1^h\,,\\
\frac{\hat y_i^h - hx_i}{h^2} &= \frac{v_i}{h} + wx_i^\bot + \psi_i^h\,,\quad \text{for }i=2,3\,,
\end{align*}
where
\begin{align*}
\psi_1^h &= u^h - u - x_2\Big(\frac{R_{21}^h}{h} - v_2'\Big) 
- x_3\Big(\frac{R_{31}^h}{h} - v_3'\Big) + \beta_1^h\,,\\
\psi_i^h &= \frac{1}{h}(v_i^h - v_i) + (w^h - w)x_i^\bot + \beta_i^h\,,\quad \text{for }i=2,3\,.
\end{align*}
Previously established convergence results imply that $(\psi_1^h,h\psi_2^h,h\psi_3^h)\to 0$ 
strongly in the $L^2$-norm. 
Moreover, the derivatives are given by
\begin{align*}
\pa_1\psi_1^h &= (u^h)' - u' - x_2\Big(\frac{(R_{21}^h)'}{h} - v_2''\Big) 
- x_3\Big(\frac{(R_{31}^h)'}{h} - v_3''\Big) + \pa_1\beta_1^h\,,\\
\pa_j^h\psi_1^h &= \frac{v_j'}{h} - \frac{R_{j1}^h}{h^2} + \pa_j^h\beta_1^h\,,\quad\text{for }j=2,3\,,\\
\pa_j^h\psi_i^h &= \frac{(-1)^j}{h}(1-\delta_{ij})(w^h - w) + \pa_j^h\beta_i^h\,,\quad\text{for }i,j=2,3\,,\\
\pa_1\psi_i^h &= \frac{1}{h}\left((v_i^h)' - v_i'\right) + \left((w^h)' - w'\right)x_i^\bot 
+ \pa_1\beta_i^h\,,\quad \text{for }i=2,3\,,
\end{align*}
which together with known convergence results immediately give $\|\sym\nabla_h\psi^h\|_{L^2(\Omega)}\leq C$. 
Invoking (\ref{eq:A}), we obtain the following representation:
\begin{equation}\label{eq:f+r2}
\frac{1}{h^2}\sym\left(\nabla_h\hat y^h - I\right) = u'e_1\otimes e_1 + \sym(\imath(A'\projyz))
+ \sym\nabla_h\psi^h\,,
\end{equation}
Additionally, 
using $(\beta_2^h(x_1,\cdot),\beta_3^h(x_1,\cdot))\in \B$ for a.e.~$x_1\in (0,L)$, one can easily 
check that 
\[
\int_\omega(x_3\psi_2^h - x_2\psi_3^h)\dd x' = -(w^h - w)\int_\omega(x_2^2 + x_3^2)\dd x' \to 0
\quad\text{strongly in }L^2\,.
\]
Next, we compute the symmetrized strain using decomposition (\ref{eq:f+r2}):
\begin{align*}
\sym G^h & = \frac{1}{h^2}\sym\paren*{ (R^h)^T\nabla_h\hat y^h - I}
= \frac{1}{h^2}\sym ((R^h - I)^T\nabla_h\hat y^h) + \frac{1}{h^2}\sym(\nabla_h\hat y^h - I)\\
& = \frac{1}{h^2}\sym((R^h - I)^T(\nabla_h\hat y^h - R^h)) - \frac{1}{h^2} \sym (R^h - I)
+ \frac{1}{h^2}\sym (\nabla_h\hat y^h - I)\\
%& =: \tilde o^h - \sym\frac{R^h - I}{h^2} + \sym\frac{\nabla_h\hat y^h - I}{h^2}\\
& := \tilde o^h - \frac{1}{h^2} \sym (R^h - I) + u'e_1\otimes e_1 + \sym(\imath(A'\projyz)) +
	\sym\nabla_h\psi^h\\
& = u'e_1\otimes e_1 + \sym(\imath(A'\projyz)) - \frac{1}{2}A^2 +
	\sym\nabla_h\psi^h + o^h\\
& =: \sym H + \sym\nabla_h\psi^h + o^h\,,
\end{align*} 
where $o^h \to 0$ strongly in $L^2(\Omega,\R^{3\times3})$, 
and $\sym H = u'e_1\otimes e_1 + \sym(\imath(A'\projyz)) - \frac12A^2$.	In this way we decomposed $\sym G^h$ into 
a fixed and relaxation part. A part of $\sym H$ can be further transferred to the relaxation 
terms as follows
\begin{align*}
\sym H &= \Big(u' + \frac12\left((v_2')^2 + (v_3')^2\Big)\right)e_1\otimes e_1 + \sym(\imath(A'\projyz)) \\
 &\qquad + \frac12\left(
 \begin{array}{ccc}
 0 & v_3'w & -v_2'w \\
 v_3'w & w^2 + (v_2')^2 & v_2'v_3'\\
 -v_2'w & v_2'v_3' & w^2 + (v_3')^2
 \end{array}
 \right)\\
 &=: \sym (\imath(m)) + \sym \nabla_h\alpha^h - \sym \imath(\pa_1\alpha^h)\,,
 %\sym \left(0\, |\, \pa_2^h \alpha^h\, |\, \pa_3^h\alpha^h\right)\,,
\end{align*}
where
\begin{equation}\label{def:m}
m = \Big(u' + \frac12((v_2')^2 + (v_3')^2)\Big)e_1 + A'\projyz\,,
\end{equation}
and
\begin{equation*}
\alpha^h(x) = h\left(\begin{array}{c}
x_2v_3'w - x_3v_2'w \\ 
\frac12x_2(w^2 + (v_2')^2) + \frac12x_3v_2'v_3'\\
\frac12x_2v_2'v_3' + \frac12x_3(w^2 + (v_3')^2)
\end{array}\right)\,.
\end{equation*}
Finally, we have decomposition
\begin{equation}\label{eq:symGh}
\sym G^h = \sym (\imath(m)) + \sym\nabla_h\psi^h + o^h\,,
\end{equation}
with updated relaxation $\psi^h$ and $L^2$--zero convergent part $o^h$.

The stress field $E^h : \Omega \to \R^{3\times3}$ is defined by 
\begin{equation*}
E^h:=\frac{1}{h^2}DW^h(\cdot,I + h^2G^h)\,.
\end{equation*}
Using the assumption (C3) on $W^h$, in particular
estimate (\ref{eq:ass}), implies that $W^h$ is differentiable a.e.~in $x\in \Omega$ and 
\begin{equation}\label{eq:bound.DW}
\forall\, G\in\R^{3\times3}\,, \ \forall\, h>0\ :\ \esssup_{x\in\Omega}|DW^h(x,I+G) - \A^h(x)G| \leq r(|G|)|G|\,,
\end{equation}
and therefore,  
\begin{equation*}
|DW^h(\cdot, I + h^2G^h)| \leq r(h^2|G^h|)h^2|G^h| + \beta h^2|G^h|\quad\text{a.e.~in }\Omega\,.
\end{equation*}
Let us denote the set
\begin{equation*}
B_h := \{x\in \Omega\ : \ h^2|G^h(x)|\leq 1\}\,,
\end{equation*}
then from the previous inequality
\begin{equation*}
|DW^h(\cdot, I + h^2G^h)| \leq Ch^2|G^h|\quad\text{pointwise in }B_h\,,
\end{equation*}
which yields
\begin{equation*}
|E^h|\leq C|G^h|\quad\text{pointwise in }B_h\,.
\end{equation*}
On the other hand on $\Omega\backslash B_h$, i.e.~on the set where 
$|G^h| > h^{-2}$ a.e., applying hypothesis (H5) we conclude
\begin{equation*}
|E^h| \leq \frac{\kappa}{h^2}\left(|I + h^2G^h| + 1 \right) 
\leq \frac{\kappa}{h^2} \left(h^2|G^h| + \sqrt{3} + 1\right)
\leq C|G^h|\quad\text{pointwise in }\Omega\backslash B_h\,.
\end{equation*}
Therefore, we have a uniform estimate on the whole set,
\begin{equation}\label{eq:Eh_pb}
|E^h| \leq C|G^h|\,
\quad\text{pointwise in }\Omega\,,
\end{equation}
which together with the uniform $L^2$-bound for the strain sequence $(G^h)$ implies the uniform
$L^2$-bound on $(E^h)$ and consequently the weak convergence (on a subsequence)
\begin{equation}
E^h \rightharpoonup E\quad\text{ in }L^2(\Omega,\R^{3\times3})\,.
\end{equation}

\subsection{Representation of elastic energy functionals}\label{sec:energy_rep}
In this subsection we briefly recall a variational approach for general (non-periodic) simultaneous homogenization 
and dimension reduction in the framework of three-dimensional nonlinear elasticity theory. 
This approach has been 
thoroughly undertaken in case of von K\'arm\'an plate \cite{Vel15} and bending rod \cite{MaVe15}, while 
the linear plate model has been outlined in \cite{BuVe15}. The theorem on geometric rigidity 
provides a decomposition of the symmetrized strain to a sum of a fixed and relaxation part (cf.~previous section). 
Utilizing the corresponding Griso's decomposition \cite{Gri05,Grisodec08} gives a further characterization 
of the relaxation part, which enables an operational representation of the elastic energies 
(cf.~Lemma \ref{lemma:key} below), suitable for the 
application of appropriate $\Gamma$-convergence techniques to eventually identify the limiting elastic energy. 
  
In the following we only provide basic steps of the method and state the final results. To start with,
let us define so called lower and upper $\Gamma$-limits. For a monotonically decreasing zero sequence of positive
numbers $(h)\subset (0,+\infty)$, $m\in L^2(\Omega,\R^3)$ and an open set $O\subset (0,L)$, we define: 
\begin{align*}
\Khm(m, O) &= \inf\Big\{\liminf_{h\downarrow0}\int_{O\times \omega}
		Q^{h}(x,\sym\imath(m) + \sym\nabla_{h}\psi^{h})\dd x\ | \\
 		& \qquad\ ~\quad(\psi_1^{h},h\psi_2^{h},{h}\psi_3^{h}) \to 0 
 		\text{ in }L^2(O\times \omega,\R^3)\,,\ \tw(\psi_2^{h},\psi_3^{h}) \to 0
 		\text{ in }L^2(O)\Big\}\,;\nonumber
\end{align*}
\begin{align*}
\Khp(m, O) &= \inf\Big\{\limsup_{h\downarrow0}\int_{O\times \omega} 
		Q^{h}(x,\sym\imath(m) + \sym\nabla_{h}\psi^{h})\dd x\ | \\
 		& \qquad\ ~\quad(\psi_1^{h},h\psi_2^{h},{h}\psi_3^{h}) \to 0 
 		\text{ in }L^2(O\times \omega,\R^3)\,,\ \tw(\psi_2^{h},\psi_3^{h}) \to 0 
 		\text{ in }L^2(O)\Big\}\,.\nonumber
\end{align*}
\noindent The above infimization is taken over all sequences 
$({\psi}^h)\subset H^1(O\times\omega,\R^3)$ such that\\
$(\psi_1^{h},h\psi_2^{h},{h}\psi_3^{h}) \to 0$ and twist functions $\tw(\psi_2^{h}, \psi_3^{h})\to 0$ 
strongly in the $L^2$-topology as $h\to0$.
The identical proof to the one presented for Lemma 3.4 in \cite{Vel15} gives the continuity of $\Khm$ and $\Khp$ with
respect to the first variable. Utilizing the diagonal procedure yields the equality of $\Khm$ and $\Khp$
for a subsequence, still denoted by $(h)$, on $L^2(\Omega,\R^3)\times \mathcal O$, where $\mathcal O$ denotes a 
countable family of open subsets of $(0,L)$. This asserts the definition of the functional
\begin{equation}\label{2.def:Kh_O}
\Kh(m,O) := \Khm(m,O) = \Khp(m,O)\,, 
\quad \forall\, m\in L^2(\Omega,\R^3)\,,\ 
\ \forall\, O\in \mathcal O\,.
\end{equation}
Adopting the strategy developed in \cite[cf.~Lemma 2.10]{MaVe15} and \cite[cf.~Lemma 3.8]{Vel15} one can prove
the following key lemma.
\begin{lemma}\label{lemma:key}
Let $(h)\subset (0,+\infty)$, $h\downarrow0$, be a sequence of positive numbers which satisfies
(\ref{2.def:Kh_O}) for every open set $O\subset (0,L)$.
Then there exists a subsequence, still denoted by $(h)$, which satisfies that for every 
$m\in L^2(\Omega,\R^3)$ there exists $({\psi}^{h})\subset H^1(\Omega,\R^3)$ 
such that for every open subset $O\subset (0,L)$, we have
\begin{equation}\label{2:eq.minseq}
\Kh(m, O) = \lim_{h\downarrow0}\int_{O\times \omega} 
						Q^{h}(x,\sym \imath(m)  + \sym\nabla_{h}\psi^{h})\dd x\,,
\end{equation}
and the following properties hold:
\begin{enumerate}[(a)]
  \item $(\psi^{h}_1,h\psi^{h}_2,h\psi_3^{h})\to 0$ and $\tw(\psi_2^{h}, \psi_3^{h})\to 0$ 
  		strongly in the $L^2$-norm as $h\downarrow0$.
  \item The sequence $(|\sym\nabla_{h}{\psi}^{h}|^2)$ is equi-integrable and 
  		there exist sequences\\ $(\Psi^{h})\subset H^1((0,L),\R^{3\times3}_{\skw})$ and
  		$(\vartheta^{h})\subset H^1(\Omega,\R^3)$ satisfying: $\Psi^{h}\to 0$, $\vartheta^{h}\to0$ strongly in 
  		the $L^2$-norm, and
  		\begin{equation*}
  			\sym \nabla_{h}{\psi}^{h} = \sym\imath((\Psi^{h})'\projyz)
  				+ \sym\nabla_{h}\vartheta^{h}\,.
  		\end{equation*}
  		Moreover, $(|(\Psi^{h})'|^2)$ and $(|\nabla_{h}\vartheta^{h}|^2)$
  		are equi-integrable and the following inequality holds
  		\begin{align*}
  		\limsup_{h\downarrow0}\left(\|\Psi^{h}\|_{H^1(O)} + \|\nabla_{h}\vartheta^{h}\|_{L^2(O\times\omega)}
  		\right) \leq C(\beta\|m\|_{L^2(O\times\omega)}^2 + 1)\,, 
  		\end{align*}
 for some $C>0$ independent of $O\subset (0,L)$.
 \item \emph{(orthogonality)} If $(\vphi^{h})\subset H^1(\Omega,\R^3)$ is any other sequence that satisfies (a) 
  		and $(\sym\nabla_{h}\vphi^{h})$ is bounded in $L^2(\Omega,\R^{3\times3})$, then
  		\begin{equation}\label{eq:orth_prop_2}
  			\lim_{h\downarrow0}\int_\Omega\A^{h}(\sym\imath(m) + \sym\nabla_{h}\psi^{h})
  			:\sym\nabla_{h}\vphi^{h} \dd x = 0\,.
  		\end{equation}
  \item \emph{(uniqueness)} If $(\vphi^{h})\subset H^1(\Omega,\R^3)$ is any other sequence that satisfies 
  		(\ref{2:eq.minseq}) and (a), then
  		\begin{equation*}
  			\|\sym \nabla_{h}\psi^{h} - \sym\nabla_{h}\vphi^{h} \|_{L^2(\Omega)} \to 0\,,
  		\end{equation*}
  		and $(|\nabla_{h}\vphi^{h}|^2)$ is equi-integrable.
\end{enumerate}
\end{lemma}
\noindent An important feature of the method is the localization property of the relaxation sequence,
i.e.~if we know the relaxation sequence for the interval $(0,L)$, the relaxation sequence for an arbitrary open
subset $O\subset (0,L)$ and fixed $m\in L^2(\Omega,\R^3)$, is simply obtained by restriction. 

Finally, we provide the integral representation of the functional $\Kh$ 
(cf.~\cite[Proposition 2.12]{MaVe15}). Recall from (\ref{def:m}) that
$m$ is of the form $m = (u' + \frac12((v_2')^2 + (v_3')^2)e_1 + A'\projyz$. Therefore, we consider the mapping
$m: L^2(0,L) \times L^2((0,L),\R_{\skw}^{3\times3}) \to L^2(\Omega,\R^3)$ defined by 
$m(\varrho,\Psi) = \varrho e_1 + \Psi \projyz$.
\begin{proposition}\label{prop:intrep}
Let $(h)\subset (0,+\infty)$ be a sequence monotonically decreasing to zero. Then there exists a subsequence
still denoted by $(h)$ and a measurable function $Q^0:(0,L)\times\R\times \R^{3} \to \R$, depending
on $(h)$, such that for every open subset $O\subset (0,L)$ and every 
$(\varrho,\Psi)\in L^2(0,L)\times L^2((0,L),\R_{\skw}^{3\times3})$ we have
\begin{equation} \label{matt1} 
\Kh(m(\varrho,\Psi),O) = \int_O Q^0(x_1,\varrho(x_1),\axl \Psi(x_1))\dd x_1\,.
\end{equation}
Moreover, for a.e.~$x_1\in (0,L)$, $Q^0(x_1,\cdot,\cdot):\R^4 \to \R$ is a bounded and coercive quadratic form.
\end{proposition}
At this point we also define function $Q_1^0:(0,L)\times\R^3 \to\R$ by
\begin{equation*}
Q_1^0(x_1, v) 
= \min_{z\in\R}Q^0(x_1,z, v)\quad\text{for all }v\in\R^{3}\ \text{ and a.e. }x_1\in (0,L)\,,
\end{equation*}
and function $\varrho_0:(0,L)\times\R^{3}\to\R$ satisfying 
$Q_1^0(x_1,\axl F) = Q^0(x_1,\varrho_0(x_1,\axl F),\axl F)$
for all $F\in\R_{\skw}^{3\times3}$ and a.e.~$x_1\in(0,L)$. One can also prove that 
$Q_1^0(x_1,\cdot)$ is a bounded and coercive quadratic form for a.e.~$x_1\in(0,L)$. 
The linear operators associated with the quadratic forms $Q^0(x_1,\cdot,\cdot)$ and $Q^0(x_1,\cdot)$ are denoted 
by $\mathbb{A}^0(x_1)$ and $\mathbb{A}_1^0(x_1)$, respectively. 
\subsection{Variational derivative of the limit elastic energy}\label{sec:varder}
Let $(h)\subset (0,+\infty)$ be a monotonically decreasing zero  sequence and
let $m\in L^2(\Omega,\R^3)$ be given. According to Lemma \ref{lemma:key}, there exist a subsequence 
still denoted by $(h)$ and a relaxation sequence $(\psi^h_m) \subset H^1(\Omega,\R^3)$ depending on $m$, 
satisfying $(\psi_{m,1}^h,h\psi_{m,2}^h,h\psi_{m,3}^h)\to0$ and $\tw(\psi_{m,2}^h,\psi_{m,3}^h)\to 0$ 
strongly in the $L^2$-norm, such that the limit elastic energy $\Kh(m):=\Kh(m,(0,L))$ is given by
\begin{align*}
\Kh(m) & = \lim_{h\downarrow0}\int_{\Omega} 
						Q^{h}(x,\sym\imath(m)  + \sym\nabla_{h}\psi^{h}_m)\dd x\\
&= \lim_{h\downarrow0}\frac12\int_\Omega \A^h(\sym\imath(m) + \sym\nabla_h\psi^h_m):
(\sym\imath(m) + \sym\nabla_h\psi^h_m)\dd x\,.
\end{align*}
In the following we compute the variational derivative of $\Kh$ at the point $m$.
Let $n\in C^\infty(\Omega,\R^3)$, such that $n(0,x') = 0$ for all $x'\in\Omega$, be a test function.
Then by definition
\begin{equation}\label{eq:varK}
\frac{\delta \K(m)}{\delta m}[n] = \lim_{\eps\downarrow0}\frac{\K(m + \eps n) - \K(m)}{\eps}\,.
\end{equation}
With a trick of successive adding of the corresponding relaxation sequences and using the orthogonality property
(\ref{eq:orth_prop_2}), for a suitable subsequence of $(h)$ we calculate:
\begin{align*}
\Kh(m &+ \eps n) - \Kh(m)\\
&= \lim_{h\downarrow0}\frac12\int_\Omega \A^h(\sym\imath(m + \eps n) + \sym\nabla_h\psi^h_{m + \eps n}):
(\sym\imath(m + \eps n) + \sym\nabla_h\psi^h_{m + \eps n})\dd x\\
&\quad - \lim_{h\downarrow0}\frac12\int_\Omega \A^h(\sym\imath(m) + \sym\nabla_h\psi^h_m):
(\sym\imath(m) + \sym\nabla_h\psi^h_m)\dd x\\
& = \lim_{h\downarrow0}\frac12\int_\Omega \A^h(\sym\imath(m + \eps n) + \sym\nabla_h\psi^h_{m + \eps n}):
\sym\imath(m + \eps n) \dd x\\
&\quad - \lim_{h\downarrow0}\frac12\int_\Omega \A^h(\sym\imath(m) + \sym\nabla_h\psi^h_m):
\sym\imath(m)\dd x\,\\
& = \lim_{h\downarrow0}\frac12\int_\Omega \A^h(\sym\imath(m + \eps n) + \sym\nabla_h\psi^h_{m + \eps n}):
(\sym\imath(m ) + \sym\nabla_h\psi^h_m) \dd x\\
& \quad+ \lim_{h\downarrow0}\frac\eps2\int_\Omega \A^h(\sym\imath(m + \eps n) 
+ \sym\nabla_h\psi^h_{m + \eps n}):(\sym\imath(n) + \sym\nabla_h\psi^h_n) \dd x\\
&\quad - \lim_{h\downarrow0}\frac12\int_\Omega \A^h(\sym\imath(m) + \sym\nabla_h\psi^h_m):
\sym\imath(m)\dd x\\
& = \lim_{h\downarrow0}\frac12\int_\Omega \A^h(\sym\imath(m) + \sym\nabla_h\psi^h_m):
\sym\imath(m + \eps n) \dd x\\
& \quad+ \lim_{h\downarrow0}\frac\eps2\int_\Omega \A^h (\sym\imath(n) 
+ \sym\nabla_h\psi^h_n): \sym\imath(m + \eps n)  \dd x\\
&\quad - \lim_{h\downarrow0}\frac12\int_\Omega \A^h(\sym\imath(m) + \sym\nabla_h\psi^h_m):
\sym\imath(m)\dd x\\
&=\lim_{h\downarrow0}\eps\int_\Omega \A^h (\sym\imath(m) + \sym\nabla_h\psi^h_m): \sym\imath(n)  \dd x\\ 
&\quad+\lim_{h\downarrow0}\frac{\eps^2}{2}\int_\Omega \A^h (\sym\imath(n) + \sym\nabla_h\psi^h_n): \sym\imath(n)  \dd x.
\end{align*}
Finally, according to the definition (\ref{eq:varK}) and utilizing the uniform $L^\infty$-bound 
for the sequence of tensors $(\A^h)$, we infer
\begin{equation}\label{eq:varder}
\frac{\delta \Kh(m)}{\delta m}[n] = \lim_{h\downarrow0}\int_\Omega \A^h(\sym\imath(m) 
+ \sym\nabla_h\psi^h_m):\sym\imath(n)\dd x\,.
\end{equation}

%%%%%%%%%%%%%%%%%%%%%%%%%%%%%%%%%%%%%%%%%%%%%%%%%%%%%%%%%%%%%%%%%%%
\section{Derivation of homogenized Euler--Lagrange equations --- proof of Theorem \ref{thm:main}}
\label{sec:homEL}

Taking the $L^2$-variation of the energy functional $\mathcal E^h$ defined by (\ref{1.eq:energy}),
one finds the Euler--Lagrange equation in the weak form:
\begin{equation}\label{eq:EL1}
\frac{\delta \mathcal E^h(y^h)}{\delta y^h}[\phi] 
= \int_\Omega\left(DW^h(x,\nabla_h y^h):\nabla_h\phi - h^3(f_2\phi_2 + f_3\phi_3)\right)\dd x = 0\,,
\end{equation}
for all test functions $\phi\in H_\omega^1(\Omega,\R^3)$. 
Let $\hat{y}^h$ be a stationary point of $\mathcal E^h$, i.e.~it satisfies (\ref{eq:EL1}). 
From the frame indifference of $W^h$ it follows that $R^TDW^h(x, RF) = DW^h(x,F)$ for all $R\in\SO(3)$,
$F\in\R^{3\times3}$ and a.e.~$x\in\Omega$, which implies (using that $\nabla_h\hat{y}^h = R^h(I + h^2G^h)$)
\begin{equation}\label{eq:DWDy}
DW^h(x,\nabla_h\hat{y}^h) = R^hDW^h(x, I + h^2G^h) = h^2R^hE^h\,.
\end{equation}
Taylor expansion around the identity gives
\begin{equation*}
DW^h(x, I + h^2G^h) = h^2G^h + \zeta^h(x,h^2G^h)\,,
\end{equation*}
where $\zeta^h$ is such that %$|\zeta^h(\cdot,F)|/|F| \to 0$ as $|F|\to0$ a.e.~in $\Omega$,
$|\zeta^h(\cdot,F)|/|F|\leq r(|F|)$ uniformly in $\Omega$, for all $F\in\R^{3\times3}$ and $h>0$. 
The latter follows from the assumption (\ref{eq:ass}) on admissible composite materials.
Since $D^2 W^h(x,I) = \A^h(x)$ and $\A^h(x)$ is a symmetric tensor, the above identity yields
\begin{equation}\label{eq:Edec}
E^h(x) = \A^h(x)\sym G^h(x) + \frac{1}{h^2}\zeta^h(x,h^2G^h)\,,
\end{equation}
which after employing (\ref{eq:symGh}) leads to
\begin{equation}\label{eq:s+r}
E^h = \A^h(\sym\imath(m) + \sym\nabla_h\psi^h) + 
\frac{1}{h^2}\zeta^h(\cdot,h^2G^h) + \A^h o^h\,.
\end{equation}

\subsection{Orthogonality property}
In order to identify the fixed part $m$ of the symmetrized strain as a stationary point of the limit
energy, we first prove the following result.
\begin{lemma}\label{thm:orth_prop}
Let $(\tensor{A}^h)$ be a sequence of tensors describing an admissible composite material, let 
$m$ be the fixed part of the symmetrized strain defined by (\ref{def:m}), and
$(\psi^h)\subset H^1(\Omega,\R^3)$ the corresponding relaxation sequence 
satisfying $(\psi_1^h,h\psi_2^h,h\psi_3^h)\to0$,
$\tw(\psi_2^h,\psi_3^h)\to 0$ strongly in the $L^2$-norm and 
$\|\sym \nabla_h\psi^h\|_{L^2(\Omega)} \leq C$. Then, for every 
sequence $(\vphi^h)\subset H^1(\Omega,\R^3)$ satisfying $(\vphi_1^h,h\vphi_2^h,h\vphi_3^h)\to0$,
$\tw(\vphi_2^h,\vphi_3^h)\to 0$
strongly in the $L^2$-norm and $(|\sym \nabla_h\vphi^h|^2)$ is equi-integrable, the following
\emph{orthogonality property} holds
\begin{equation}\label{eq:orth_prop}
\lim_{h\downarrow0}\int_\Omega\A^h(\sym \imath(m) + \sym\nabla_h\psi^h):\sym\nabla_h\vphi^h\dd x = 0\,.
\end{equation}
\end{lemma}
\begin{proof}
Let $(\psi^h)\subset H^1(\Omega,\R^3)$ and $(\vphi^h)\subset H^1(\Omega,\R^3)$ be arbitrary sequences  
satisfying the assumptions of the theorem. Applying the Griso's decomposition to the sequence 
$(\vphi^h)$ (cf.~\cite[Corollary 2.3]{MaVe15}), there exist sequences $(\Phi^h)\subset H^1((0,L), \R^3_{\skw})$, 
$(\phi^h)\subset H^1(\Omega,\R^3)$ and $(o^h)\subset L^2(\Omega,\R^{3\times3})$ satisfying:
\begin{equation*}
\sym \nabla_h\varphi^h = \sym \imath((\Phi^h)'\projyz) + \sym\nabla_h \phi^h + o^h\,,
\end{equation*}
$\Phi^h\to 0$, $\phi^h\to 0$, $o^h\to 0$ strongly in the $L^2$-norm, and
\begin{equation}\label{eq:griso_bound}
\|\Phi^h\|_{H^1(0,L)} + \|\phi^h\|_{L^2(\Omega)} + \|\nabla_h\phi^h\|_{L^2(\Omega)} 
\leq C\|\sym\nabla_h\varphi^h\|_{L^2(\Omega)}\,,\quad\forall h>0\,.
\end{equation}
Furthermore, there exist subsequences $(\Phi^h)$ and $(\phi^h)$ (still denoted by $(h)$) and sequences
$(\tilde\Phi^h)\subset H^1((0,L),\R^3)$ and $(\tilde\phi^h)\subset H^1(\Omega,\R^3)$ such that 
$|\{\Phi^h \neq \tilde\Phi^h\}\cup \{(\Phi^h)'\neq (\tilde\Phi^h)'\}|\to0$ and 
$|\{\phi^h \neq \tilde\phi^h\}\cup \{\nabla\phi^h \neq \nabla\tilde\phi^h\}|\to0$ as $h\downarrow0$, and
the sequences $(|(\tilde\Phi^h)'|^2)$ and $(|\nabla_h\tilde\phi^h|^2)$ are equi-integrable (cf.~\cite{FMP98} and
\cite[Lemma 2.17]{MaVe15}). 
The rest of the proof will be divided into two parts showing the property (\ref{eq:orth_prop}) 
for sequences $(\tilde\phi^h)$ and $(\tilde\Phi^h)$, respectively. For ease of presentation, we will
in future denote these sequences again by $(\phi^h)$ and $(\Phi^h)$.

{\noindent\em Part 1.} The equi-integrability property of the sequence $(\phi^h)$ 
allows us to modify each $\phi^h$ to zero near the boundary (cf.~\cite[Lemma 3.6]{Vel15}), thus, making it
an eligible test function in the Euler--Lagrange equation (\ref{eq:EL1}).
Using the identity (\ref{eq:s+r}) and the modified $\phi^h$ as a test function
in the Euler--Lagrange equation (\ref{eq:EL1}), after division by $h^2$, we obtain (according to \eqref{eq:DWDy}) 
\begin{align*}
\int_\Omega R^h\A^h & (\sym \imath(m) + \sym\nabla_h\psi^h):\nabla_h\phi^h\dd x 
= \int_\Omega R^h\left(E^h - \frac{1}{h^2}\zeta^h(\cdot,h^2G^h) - \A^h o^h\right):\nabla_h\phi^h\dd x\\
& = \int_\Omega h(f_2\phi^h_2 + f_3\phi_3^h) - 
	\int_\Omega R^h\left(\frac{1}{h^2}\zeta^h(\cdot,h^2G^h) + \A^h o^h\right):\nabla_h\phi^h\dd x\,.
\end{align*}
Obviously, the first and the last term converge to $0$ as $h\downarrow0$. Let us examine the 
second term
\begin{equation*}
\frac{1}{h^2}\int_\Omega R^h\zeta^h(\cdot,h^2G^h) :\nabla_h\phi^h\dd x\,.
\end{equation*}
Denote the set $S_h^\alpha := \{x\in \Omega\ : \ h^2|G^h(x)|\leq h^{\alpha}\}$ for some $0 < \alpha < 2$.
On $S_h^\alpha$ we have 
\begin{equation*}
\frac{|\zeta^h(\cdot,h^2G^h)|}{h^2|G^h|}|G^h| 
\leq \sup\left\{\frac{|\zeta^h(\cdot,h^2\tilde G^h)|}{h^2|\tilde G^h|} 
\ : \ h^2|\tilde G^h|\leq h^\alpha\right \}|G^h|
\leq r(h^\alpha)|G^h|\,.
\end{equation*}
Therefore, 
\begin{align*}
\frac{1}{h^2}\left| \int_{S_h^\alpha} R^h\zeta^h(\cdot,h^2G^h) :\nabla_h\phi^h\dd x \right| 
\leq r(h^\alpha)\|R^h\|_{L^\infty(\Omega)}\|G^h\|_{L^2(\Omega)}\|\nabla_h\phi^h\|_{L^2(\Omega)}
\leq Cr(h^\alpha)\to0\,,
\end{align*}
as $h\downarrow0$.
On the other hand, on $\Omega\backslash S_h^\alpha$ we have a pointwise a.e.~bound
\begin{equation*}
\frac{1}{h^2}|\zeta^h(\cdot,h^2G^h)| \leq C|G^h|\quad\text{a.e.~on }\Omega\backslash S_h^\alpha\,,
\end{equation*} 
which in fact holds pointwise a.e.~on $\Omega$. This follows by the traingle inequality from (\ref{eq:Edec})
using (\ref{eq:Eh_pb}) and $|\A^h(x)G^h(x)|\leq \beta|G^h(x)|$ for a.e.~$x\in\Omega$. 
Therefore, using the H\"older and Chebyshev inequalities, respectively, we find
\begin{align*}
\frac{1}{h^2}\left| \int_{\Omega\backslash S_h^\alpha} R^h\zeta^h(\cdot,h^2G^h) :\nabla_h\phi^h\dd x \right| 
& \leq C\int_{\Omega\backslash S_h^\alpha} |G^h||\nabla_h\phi^h|\dd x \\
& \leq C \|\nabla_h\phi^h\|_{L^\infty(\Omega\backslash S_h^\alpha)}\int_{\Omega\backslash S_h^\alpha} |G^h|\dd x\\
& \leq C \|\nabla_h\phi^h\|_{L^\infty(\Omega)}|\Omega\backslash S_h^\alpha|^{1/2}\\
& \leq C\|\nabla_h\phi^h\|_{L^\infty(\Omega)}h^{1-\alpha/2}\,.
\end{align*}
In order to successfully pass to the limit when $h\downarrow0$, we have to again replace the sequence $(\phi^h)$ by
a sequence obtained by means of Corollary \ref{cor:fmp2} (cf.~Lemma \ref{lema:fmp1}). 
We choose a strictly increasing 
sequence $(s_h)\subset (0,+\infty)$ defined by $s_h = h^{-\gamma}$ with a constant $\gamma >0$ 
satisfying $1 - \alpha/2 - \gamma > 0$. The obtained sequence $(\tilde{\phi}^h)$ then satisfies 
$\|\nabla_h\tilde\phi^h\|_{L^\infty(\Omega)} \leq Cs_h$, and we infer
\begin{equation*}
\lim_{h\downarrow0}\frac{1}{h^2}
\left| \int_{\Omega\backslash S_h^\alpha} R^h\zeta^h(\cdot,h^2G^h) :\nabla_h\tilde\phi^h\dd x \right| = 0\,.
\end{equation*}
Thus, we have shown
\begin{equation*}
\lim_{h\downarrow0}\int_\Omega R^h\A^h(\sym \imath(m) + \sym\nabla_h\psi^h):\nabla_h\tilde\phi^h\dd x = 0\,.
\end{equation*}
Since $R^h \to I$ strongly in the $L^\infty$-norm, it follows that
\begin{equation*}
\lim_{h\downarrow0}\int_\Omega \A^h(\sym \imath(m) + \sym\nabla_h\psi^h):\nabla_h\tilde\phi^h\dd x = 0\,,
\end{equation*}
while the symmetry property of $\A^h(\sym \imath(m) + \sym\nabla_h\psi^h)$ eventually implies
\begin{equation}\label{eq:conv2}
\lim_{h\downarrow0}\int_\Omega \A^h(\sym \imath(m) + \sym\nabla_h\psi^h):\sym \nabla_h\tilde\phi^h\dd x = 0\,.
\end{equation}
Due to the fact that 
$\{\phi^h = \tilde\phi^h\} = \{\phi^h = \tilde\phi^h, \nabla_h\phi^h = \nabla_h\tilde\phi^h\}\cup \set N$ 
\cite[Theorem 3, Sec.~6]{EG92},
where $\set N$ is a set of measure zero, and $|\{\phi^h \neq\tilde\phi^h\}|\to 0$ as $h\to0$, we deduce
\begin{align*}
\lim_{h\downarrow0}\int_\Omega \A^h(\sym \imath(m) & + \sym\nabla_h\psi^h):\sym \nabla_h\phi^h\dd x  \\
& \quad = \lim_{h\downarrow0}\int_\Omega \A^h(\sym \imath(m) + \sym\nabla_h\psi^h):\sym \nabla_h\tilde\phi^h\dd x 
= 0\,.
\end{align*}

{\em\noindent Part 2.}
Again, the equi-integrability property of the sequence $(\Phi^h)$ 
allows us to modify each $\Phi^h$ to zero near the boundary, thus, making the following functions
\begin{align}\label{def:hphi}
\hat\phi^h(x) = \Big(\Phi^h_{12}(x_1)x_2 +  \Phi^h_{13}(x_1)x_3\,, 
-\frac{1}{h}&\int_0^{x_1}\Phi_{12}^h(s)\dd s + \Phi_{23}^h(x_1)x_3\,,\\
& -\frac{1}{h}\int_0^{x_1}\Phi_{13}^h(s)\dd s - \Phi_{23}^h(x_1)x_2 \Big)\,,\nonumber
\end{align}
eligible test functions in the Euler--Lagrange equation (\ref{eq:EL1}).
One easily calculates
\begin{align*}
\sym\nabla_h\hat\phi^h & = 
\left(
\begin{array}{ ccc }
(\Phi^h_{12})'(x_1)x_2 + (\Phi^h_{13})'(x_1)x_3 & \frac12(\Phi^h_{23})'(x_1)x_3 & -\frac12(\Phi^h_{23})'(x_1)x_2 \\
\frac12(\Phi^h_{23})'(x_1)x_3  & 0 & 0 \\
-\frac12(\Phi^h_{23})'(x_1)x_2 & 0 & 0
\end{array}
\right)\\ 
& = \sym \imath((\Phi^h)'\projyz)\,.
\end{align*}
Using $\hat{\phi}^h$ as a test function in (\ref{eq:EL1}) together with
 the symmetry property of the matrix $DW^h(\cdot, F)F^T$, we obtain
\begin{align*}
\frac{1}{h^2}\int_\Omega DW^h & (x, R^h(I + h^2G^h)):\nabla_h\hat{\phi}^h\dd x \\ 
&= \frac{1}{h^2}\int_\Omega R^hDW^h(x, I + h^2G^h)(I + h^2G^h)^T(R^h)^T : \sym\nabla_h\hat{\phi}^h\dd x\\
&\quad -\frac{1}{h^2}\int_\Omega R^hDW^h(x, I + h^2G^h)\left((R^h)^T - I + h^2(G^h)^T(R^h)^T\right) 
: \nabla_h\hat{\phi}^h\dd x\\
& = \int_\Omega R^hE^h(I + h^2G^h)^T(R^h)^T : \sym \imath((\Phi^h)'\projyz) \dd x \\ 
&\quad -\int_\Omega R^hE^h\left(\frac{1}{h}((R^h)^T - I) + h(G^h)^T(R^h)^T\right) 
: h\nabla_h\hat{\phi}^h\dd x\,.
\end{align*}
Therefore, the Euler--Lagrange equation becomes
\begin{align}
\int_\Omega R^hE^h&(I + h^2G^h)^T(R^h)^T : \sym \imath((\Phi^h)'\projyz)\nonumber \dd x\\ 
&= \int_\Omega R^hE^h\left(\frac{1}{h}((R^h)^T - I) + h(G^h)^T(R^h)^T\right) 
: h\nabla_h\hat{\phi}^h\dd x + h\int_\Omega (f_2\hat\phi^h_2 + f_3\hat\phi^h_3)\dd x\,.\label{eq:EL2}
\end{align}
Since $(h\hat\phi^h_2,h\hat\phi^h_3)\to 0$ strongly in the $L^2$-norm, the force term vanishes at the limit. 
According to (\ref{eq:griso_bound}), $\|(\Phi^h)'\|_{L^2(0,L)}$ is uniformly bounded 
implying the strong convergence $h\nabla_h\hat{\phi}^h \to 0$ in the $L^2$-norm,
therefore, 
\begin{equation*}
\lim_{h\downarrow0}\frac{1}{h}\int_\Omega R^hE^h((R^h)^T - I)
: h\nabla_h\hat{\phi}^h\dd x = 0\,.
\end{equation*}
In order to infer zero at the limit as $h\downarrow0$ for the remaining term on the right-hand side in (\ref{eq:EL2}), namely
$$h\int_\Omega R^hE^h(G^h)^T(R^h)^T : h\nabla_h\hat{\phi}^h\dd x\,, $$ we
need to replace the sequence $(\Phi^h)$ with the one obtained by means of Lemma \ref{lema:fmp1}.
We take the sequence $(s_h)$ as above and obtain a sequence $(\tilde\Phi^h)$ satisfying 
$\|\tilde{\Phi}^h\|_{W^{1,\infty}(0,L)} \leq Cs_h$ for some $C>0$.
The last bound together with continuous Sobolev embedding $H^1((0,L),\R^3)\hookrightarrow L^\infty((0,L),\R^3)$ 
imply $\|h\nabla_h\tilde{\hat\phi}^h\|_{L^\infty} \leq C$, % Ch^{1-\gamma}$, the remaining term (the second one) 
where, in view of (\ref{def:hphi}), notation $\tilde{\hat\phi}^h$ is self-explaining. 
From the latter we conclude that the second term on the right-hand side of (\ref{eq:EL2}) vanishes and
infer that
\begin{equation}
\lim_{h\downarrow0}\int_\Omega R^hE^h (I + h^2G^h)^T(R^h)^T : \sym \imath((\tilde\Phi^h)'\projyz)\dd x 
= 0\,.
\end{equation}
Obviously, 
\begin{equation*}
\lim_{h\downarrow0}\int_\Omega h^2R^hE^h (G^h)^T(R^h)^T : \sym \imath((\tilde\Phi^h)'\projyz)\dd x 
= 0\,,
\end{equation*}
and therefore,
\begin{equation}\label{eq:conv3}
\lim_{h\downarrow0}\int_\Omega R^hE^h(R^h)^T : \sym \imath((\tilde\Phi^h)'\projyz) \dd x 
= 0\,.
\end{equation}
Next, we prove that
\begin{equation}\label{eq:Eh_conv}
\lim_{h\downarrow0}\int_\Omega E^h : \sym \imath((\tilde\Phi^h)'\projyz) \dd x 
= 0\,.
\end{equation}
This follows by writing 
\begin{align*}
\int_\Omega E^h : \sym \imath((\tilde\Phi^h)'\projyz) \dd x 
= \int_\Omega \left(R^h + (I - R^h)\right)E^h \left(R^h + (I - R^h)\right)^T 
: \sym \imath((\tilde\Phi^h)'\projyz) \dd x\,,
\end{align*}
and using the convergence result (\ref{eq:conv3}) with the fact that $R^h\to I$  
strongly in the $L^\infty$-norm. Now, recall that
\begin{equation*}
\A^h(\sym \imath(m) + \sym\nabla_h\psi^h) = E^h - \frac{1}{h^2}\zeta^h(\cdot,h^2G^h) + o^h\,,
\end{equation*}
where $o^h\to0$ strongly in the $L^2$-norm. 
Using truncation arguments on the sets $S_h^\alpha$ and its complement, as in the first part of the proof, 
we conclude
\begin{equation*}
\lim_{h\downarrow0}\int_\Omega  \frac{1}{h^2}\zeta^h(\cdot,h^2G^h) : \sym \imath((\tilde\Phi^h)'\projyz) \dd x 
= 0\,.
\end{equation*}
Since $\lim_{h\downarrow0}\int_\Omega  o^h : \sym \imath((\tilde\Phi^h)'\projyz) \dd x = 0\,,$ convergence
result (\ref{eq:Eh_conv}) implies
\begin{equation}\label{eq:conv1}
\lim_{h\downarrow0}\int_\Omega \A^h(\sym \imath(m) + \sym\nabla_h\psi^h) : \sym \imath((\tilde\Phi^h)'\projyz) \dd x
= 0\,.
\end{equation}
We finalize the proof with a conclusion analogous to the one from Part 1.
\end{proof}

\subsection{Identification of the limit Euler--Lagrange equations}
Let us now more closely identify terms in the Euler--Lagrange equation (\ref{eq:EL1}) 
and consider the limit when $h\downarrow0$. 
The same reasoning as in Part 2 of the proof of
Lemma \ref{thm:orth_prop} gives, after division by $h^2$, the Euler--Lagrange equation (\ref{eq:EL1}) in the form
\begin{align}
\int_\Omega R^hE^h&(I + h^2G^h)^T(R^h)^T : \sym \nabla_h\phi\, \dd x \nonumber\\
&= \int_\Omega R^hE^h\left(\frac{1}{h}((R^h)^T - I) + h(G^h)^T(R^h)^T\right)
: h\nabla_h{\phi}\,\dd x + h\int_\Omega (f_2\phi_2 + f_3\phi_3)\dd x\, \label{eq:EL3} %
\end{align}
for all test functions $\phi\in H^1_\omega(\Omega,\R^3)$. 
The aim is now to identify the limit equation in (\ref{eq:EL3}) as $h \downarrow0$. 
Using the facts that, up to a term converging to zero strongly in the $L^2$-norm,
\begin{equation*}
E^h = \A^h(\sym \imath(m) + \sym\nabla_h\psi^h) + \frac{1}{h^2}\zeta^h(\cdot,h^2G^h)\,,
\end{equation*} 
$R^h\to I$ strongly in the $L^\infty$-norm, and 
\begin{equation*}
\lim_{h\downarrow0}\int_\Omega h^2R^hE^h (G^h)^T(R^h)^T : \sym\nabla_h\phi\,\dd x = 0\,,
\end{equation*}
the limit when $h\downarrow0$ (if it exists) of 
\begin{align*}
\int_\Omega R^hE^h&(I + h^2G^h)^T(R^h)^T : \sym\nabla_h{\phi}\,\dd x
\end{align*}
equals the limit 
\begin{align*}
\lim_{h\downarrow0}\int_\Omega \big(\A^h(\sym \imath(m) 
+ \sym\nabla_h\psi^h) + \frac{1}{h^2}\zeta^h(\cdot,h^2G^h)\big) 
: \sym\nabla_h{\phi}\,\dd x\,.
\end{align*}
The remainder term $\frac{1}{h^2}\int_\Omega\zeta^h(\cdot,h^2G^h): \sym\nabla_h{\phi}\,\dd x$ vanishes in the same 
way as in the proof of Lemma \ref{thm:orth_prop}, and on the limit as $h\downarrow0$, 
equation (\ref{eq:EL3}) reduces to 
\begin{align}
\lim_{h\downarrow0}& \int_\Omega \A^h(\sym \imath(m) 
+ \sym\nabla_h\psi^h) : \sym\nabla_h{\phi}\,\dd x\,\nonumber\\
&= \lim_{h\downarrow0} \left(\int_\Omega R^hE^h\left(\frac{1}{h}((R^h)^T - I) + h(G^h)^T(R^h)^T\right)
: h\nabla_h{\phi}\,\dd x + h\int_\Omega (f_2\phi_2 + f_3\phi_3)\dd x\right)\,.\label{eq:homredEL}
\end{align}

First, consider the test function $\phi(x) = \phi_{11}(x_1)e_1$ with $\phi_{11}$ smooth and $\phi_{11}(0)=0$. 
Since $\phi_{2} = \phi_{3} = 0$, 
 $\sym\nabla_h\phi = \phi_{11}'(x_1)e_1\otimes e_1$, and $h\nabla_h\phi\to0$ strongly in the $L^2$-norm,
(\ref{eq:homredEL}) amounts to 
\begin{equation}\label{eq:phi1}
\lim_{h\downarrow0}\int_\Omega \A^h(\sym \imath(m) 
+ \sym\nabla_h\psi^h) : \phi_{11}'(x_1)e_1\otimes e_1\,\dd x = 0\,.
\end{equation}
Next, consider test functions of the form $\phi^h_{ij}(x) = hx_j\phi_{ij}(x_1)e_i$ for
$i=1,2,3$, $j=2,3$, where $\phi_{ij}$ is smooth with $\phi_{ij}(0) = 0$. The functions $\phi_{ij}^h$ obviously satisfy 
$(\phi_{ij,1}^h, h\phi_{ij,2}^h, h\phi_{ij,3}^h)\to 0$ and $\tw(\phi_{ij,2}^h, \phi_{ij,3}^h)\to 0$ strongly in the
$L^2$-norm.
Calculating
\begin{equation*}
\sym\nabla_h\phi_{ij}^h 
= \sym(hx_j\phi_{ij}'e_i\,|\, \delta_{2j}\phi_{ij}e_i\,|\, \delta_{3j}\phi_{ij}e_i)\,,
\end{equation*}
we easily conclude from (\ref{eq:homredEL}) that
\begin{equation}\label{eq:phi_ij}
\lim_{h\downarrow0}\int_\Omega \A^h(\sym \imath(m) 
+ \sym\nabla_h\psi^h) : \phi_{ij}(x_1)e_i\otimes e_j\,\dd x = 0\,,
\end{equation}
for all $i=1,2,3$, $j=2,3$.
Finally, consider the test function given by
\begin{align*}
\phi^h(x) = \left(\Phi_{12}(x_1)x_2 + \Phi_{13}(x_1)x_3\,, 
\frac{1}{h}\int_0^{x_1}\Phi_{21}(s)\dd s + \Phi_{23}(x_1)x_3\,,
\frac{1}{h}\int_0^{x_1}\Phi_{31}(s)\dd s + \Phi_{32}(x_1)x_2 \right),
\end{align*}
where $\Phi : [0,L]\to \R^{3\times3}_{\skw}$ is smooth and $\Phi(0) = 0$.
On the right-hand side of (\ref{eq:homredEL}), using the convergence results: $R^h\to I$ strongly in the $L^\infty$-norm,
$hG^h\to0$ strongly in the $L^2$-norm, $A^h\to A$ strongly in the $L^\infty$-norm, as well as the
approximation identity (\ref{eq:s+r}) for $E^h$, we are left with
\begin{align*}
\lim_{h\downarrow0}\int_\Omega \A^h &(\sym \imath(m)  + \sym \nabla_h\psi^h)A^T : {\Phi}\,\dd x\\
 &+ \int_0^L\left(f_2(x_1)\int_0^{x_1}\Phi_{21}(s)\dd s + f_3(x_1)\int_0^{x_1}\Phi_{31}(s)\dd s \right)\dd x_1\,.
\end{align*}
Let us now consider the first term of the obtained expression. Due to the real matrix identity
$XY:Z = -X:ZY$, for $Y$ being skew-symmetric matrix, the first term equals (up to a minus sign)
\begin{align*}
\lim_{h\downarrow0}\int_\Omega \A^h &(\sym \imath(m)  + \sym \nabla_h\psi^h) : {\Phi}A\,\dd x\,,
\end{align*}
and since the first matrix is symmetric, the latter in fact equals to
\begin{align}\label{eq:leftover}
\lim_{h\downarrow0}\int_\Omega \A^h &(\sym \imath(m)  + \sym \nabla_h\psi^h) : \sym({\Phi}A)\dd x\,.
\end{align}
The matrix $\Phi A$ can be explicitly computed, and its symmetric part is given by 
\begin{equation*}
\sym(\Phi A) = 
\left(
\begin{array}{ccc}
 \Phi_{12}v_2' + \Phi_{13}v_3' & \frac12(\Phi_{23}v_3' + \Phi_{13}w) & -\frac12(\Phi_{23}v_2' + \Phi_{12}w) \\
 \frac12(\Phi_{23}v_3' + \Phi_{13}w) & \Phi_{12}v_2' + \Phi_{23}w & \frac12(\Phi_{13}v_2' + \Phi_{12}v_3')\\
 -\frac12(\Phi_{23}v_2' + \Phi_{12}w) & \frac12(\Phi_{13}v_2' + \Phi_{12}v_3') & \Phi_{13}v_3' + \Phi_{23}w
 \end{array}
\right)\,.
\end{equation*}
Defining the sequence of test functions $(\varphi^{h}_A)$ by
\begin{align*}
 \varphi^h_A(x) =  
 \left( \begin{array}{c} 
 			\Phi_{23}v_3' + \Phi_{13}w \\ \Phi_{12}v_2' + \Phi_{23}w \\ \Phi_{13}v_2' + \Phi_{12}v_3'
 		\end{array} 
 \right) + hx_3 
 \left(
 		\begin{array}{c} 
 		- \Phi_{23}v_2' - \Phi_{12}w \\ \Phi_{13}v_2' + \Phi_{12}v_3' \\ \Phi_{13}v_3' + \Phi_{23}w
 		\end{array}
 \right)
 \,,
\end{align*}
it is straightforward to check that
\begin{equation}\label{eq:symPhiA}
\sym (\Phi A) = \sym\nabla_h\varphi^h_A + (\Phi_{12}v_2' + \Phi_{13}v_3')e_1\otimes e_1
  + o^h\,,
\end{equation}
where $o^h$ converges to zero strongly in the $L^2$-norm as $h\downarrow0$.
Observe that the sequence of test functions $(\vphi^h_A)$ 
satisfies $(\vphi^h_{A,1}, h\vphi^h_{A,2}, h\vphi^h_{A,3})\to 0$ and
$\tw(\vphi^h_{A,2},\vphi^h_{A,3})\to 0$ strongly in the $L^2$-norm.
Utilizing (\ref{eq:symPhiA}) in expression (\ref{eq:leftover}), 
we confer that due to the orthogonality property (\ref{eq:orth_prop}), convergence result
(\ref{eq:phi1}) and strongly to zero convergence 
of $o^h$, these terms vanish in the limit as $h\downarrow0$. 
Since,
\begin{equation}
\sym\nabla_h\phi^h = 
\left(
\begin{array}{ ccc }
\Phi_{12}'(x_1)x_2 + \Phi_{13}'(x_1)x_3 & \frac12\Phi_{23}'(x_1)x_3 & -\frac12\Phi_{23}'(x_1)x_2 \\
\frac12\Phi_{23}'(x_1)x_3  & 0 & 0 \\
-\frac12\Phi_{23}'(x_1)x_2 & 0 & 0
\end{array}
\right) = \sym \imath((\Phi)'\projyz)\,,\label{eq:sym Phi'}
\end{equation}
the left-hand side in (\ref{eq:homredEL}) can be written as
\begin{equation}\label{eq:Phi}
\lim_{h\downarrow0}\int_\Omega \A^h(\sym \imath(m) 
+ \sym\nabla_h\psi^h) : \sym \imath(\Phi'\projyz)\,\dd x\,.
\end{equation}
Combining (\ref{eq:phi1}), (\ref{eq:phi_ij}) and (\ref{eq:Phi}), the resolved limiting 
Euler--Lagrange equation (\ref{eq:homredEL}) reads
\begin{align}
\lim_{h\downarrow0}\int_\Omega \A^h(\sym \imath(m) 
+  \sym\nabla_h\psi^h) :  \sym & 
\,\Big(\phi_{11}'e_1\otimes e_1  + \sum_{i=1, j=2}^3\phi_{ij}e_i\otimes e_j 
+ \imath( \Phi'\projyz) \Big)\,\dd x\nonumber \\
&= -\int_0^L(f_2\tilde\Phi_{12} + f_3\tilde\Phi_{13})\dd x_1\,,\label{eq:ELred2}
\end{align}
where $\tilde\Phi_{1j}(x_1) = \int_0^{x_1}\Phi_{1j}(s)\dd s$ for $j=2,3$.
Now, to conclude the proof, the obtained equation 
(neglecting the terms $\sum_{i=1, j=2}^3\phi_{ij}e_i\otimes e_j$ in the first sum due to (\ref{eq:phi_ij}))
is to be interpreted as
\begin{equation*}
\frac{\delta \Kh(m)}{\delta m}\Big[\phi_{11}'e_1\otimes e_1 + \Phi'\projyz\Big] = 
-\int_0^L(f_2\tilde\Phi_{12} + f_3\tilde\Phi_{13})\dd x_1\,.
\end{equation*}
Since $(\sym \nabla_h\psi^h)$ is bounded in the $L^2$-norm, 
according to \cite[Lemma 2.17]{MaVe15}, there exists a subsequence (still denoted by $(h)$) and sequence
$(\tilde\psi^h)$ such that $(|\sym \nabla_h\tilde\psi^h|^2)$ is equi-integrable and 
$\|\sym \nabla_h \psi^h-\sym \nabla_h \tilde\psi^h\|_{L^2(O^h)} \to 0$, where $O^h \subset \Omega$ such that
$|\Omega \backslash O^h| \to 0$.
From (\ref{eq:ELred2}) we see that the same limit equation will be obtained if we replace the relaxation
sequence $(\psi^h)$ by $(\tilde\psi^h)$. 
Let $(\psi_m^h)$ be the relaxation sequence for $m$ from Lemma \ref{lemma:key}. 
Using the coercivity of $Q^h$ and the orthogonality properties (\ref{eq:orth_prop_2}) and
(\ref{eq:orth_prop}) of both sequences $(\psi_m^h)$ and $(\tilde\psi^h)$, respectively, we find that
\begin{align*}
\alpha\|\sym \nabla_h(\psi_m^h - \tilde\psi^h)\|_{L^2}^2 &\leq 
\int_\Omega Q^h(x,\sym\nabla_h(\psi_m^h - \tilde\psi^h))\dd x\\
& = \frac12\int_\Omega\A^h(\sym \imath(m) + \sym\nabla_h\psi_m^h):\sym \nabla_h(\psi_m^h - \tilde\psi^h)\dd x\\
&\quad - \frac12\int_\Omega\A^h(\sym \imath(m) + \sym\nabla_h\psi^h):\sym \nabla_h(\psi_m^h - \tilde\psi^h)\dd x\to 0
\end{align*}
as $h\downarrow0$. Therefore, we can also replace the sequence $(\tilde\psi^h)$ by $(\psi_m^h)$ 
and according to (\ref{eq:varder}), $m$ is indeed the stationary point of the limit functional $\Kh$.
Finally, since the stationarity of the point $(u,v_2,v_3,w)$ for
the functional $\mathcal E^0$ is (up to the linear force term) equivalent to the stationarity of $m$ 
(defined by (\ref{def:m})) for the functional $\Kh$, this finishes the proof of Theorem \ref{thm:main}.

In the subsequent part of the section we identify the limit Euler--Lagrange equations. 
Recalling the approximation identity (\ref{eq:s+r}), the weak convergence $E^h\rightharpoonup E$ in 
$L^2(\Omega,\R^{3\times3})$, and utilizing convergence properties for
the remainder terms, we can pass to the limit in
equation (\ref{eq:ELred2}) and obtain
\begin{equation*}
\int_\Omega E : \sym \,\Big(\phi_{11}'e_1\otimes e_1  + \sum_{i=1, j=2}^3\phi_{ij}e_i\otimes e_j 
+ \imath( \Phi'\projyz) \Big)\,\dd x 
= -\int_0^L(f_2\tilde\Phi_{12} + f_3\tilde\Phi_{13})\dd x_1\,.
\end{equation*}
In view of identity (\ref{eq:sym Phi'}), the latter equals
\begin{align}
\int_\Omega \Big(E_{11}\phi_{11}' & + \sum_{i=1, j=2}^3E_{ij}\phi_{ij}\, + x_2E_{11}\Phi_{12}'(x_1) + x_3E_{11}\Phi_{13}'(x_1) \\
&\quad + x_3E_{12}\Phi'_{23}(x_1) - x_2E_{13}\Phi'_{23}(x_1)\Big)\dd x\nonumber
 = -\int_0^L(f_2\tilde\Phi_{12} + f_3\tilde\Phi_{13} )\dd x_1\,.\label{eq:ELred3}
\end{align}
Using the moment notation (\ref{1.def:moment0})--(\ref{1.def:moments1}) 
and the fact that $\tilde\Phi_{1j}' = \Phi_{1j}$ for $j=2,3$, (\ref{eq:ELred3}) becomes
\begin{align}
\int_0^L \Big(\overline E_{11}\phi_{11}' + \sum_{i=1,j=2}^3\overline E_{ij}\phi_{ij} + \wtilde E_{11}\tilde\Phi_{12}'' + \what E_{11}\tilde\Phi_{13}'' 
& + \what E_{12}\Phi'_{23} - \wtilde E_{13}\Phi'_{23}\Big)\dd x_1 \nonumber\\
&\quad  = -\int_0^L(f_2\tilde\Phi_{12} + f_3\tilde\Phi_{13})\dd x_1\,.
\end{align}
Now by the arbitrariness of test functions, we easily derive the corresponding strong formulation
for the moments. The zeroth-order moments satisfy
\begin{align}
\overline E &= 0 \quad \text{in }(0,L)\,.
\end{align}
The first-order moments $\wtilde E_{11}$ and $\what E_{11}$ satisfy second-order boundary-value problems:
\begin{equation}
\arraycolsep=1.4pt\def\arraystretch{1.4}
\begin{array}{l}
\wtilde E_{11}''+ f_2 = 0 \quad \text{in }(0,L)\,, \\
\wtilde E_{11}(L) = \wtilde E_{11}'(L) = 0\,,
\end{array}
\end{equation}
and 
\begin{equation}
\arraycolsep=1.4pt\def\arraystretch{1.4}
\begin{array}{l}
\what E_{11}''+ f_3 = 0 \quad \text{in }(0,L)\,, \\
\what E_{11}(L) = \what E_{11}'(L) = 0\,,
\end{array}
\end{equation}
respectively. Finally, the first-order moments $\what E_{12}$ and $\wtilde E_{13}$ satisfy the first-order problem
\begin{equation}
\arraycolsep=1.4pt\def\arraystretch{1.4}
\begin{array}{l}
\what E_{12}' - \wtilde E_{13}' = 0 \quad \text{in }(0,L)\,, \\
\what E_{12}(L) = \wtilde E_{13}(L)\,.
\end{array}
\end{equation}
It remains to derive constitutive equations, which connect the moments of the limit stress with 
limit displacements and twist functions.
For $\varrho\in L^2(0,L)$ and $\Psi\in L^2((0,L),\R_{\skw}^{3\times3})$, recall the functional
\begin{align*}
\Kh(m(\varrho,\Psi)) &= \int_0^L Q^0(x_1,\varrho(x_1),\axl \Psi(x_1))\dd x_1\,,
\end{align*}
where $m(\varrho,\Psi)(x) = \varrho(x_1)e_1 + \Psi(x_1)\projyz$, and the functional
\begin{align*}
\Kh^0(\Psi) &= \int_0^L Q_1^0(x_1,\axl \Psi(x_1))\dd x_1 
= \int_0^L Q^0(x_1,\varrho_0(x_1,\axl\Psi(x_1)),\axl \Psi(x_1))\dd x_1\\
&= \Kh(m_0(\varrho_0,\Psi))\,,
\end{align*}  
where $\varrho_0:(0,L)\times \R^{3}\to \R$ is optimal for a given $\axl\Psi$. 
By Lemma \ref{lemma:key} (identity (\ref{2:eq.minseq})),
there exist sequences $(\psi^h_m)\subset H^1(\Omega,\R^3)$ and $(\psi^h_0)\subset H^1(\Omega,\R^3)$ such that:
\begin{align*}
\Kh(m(\varrho,\Psi)) &= \lim_{h\downarrow0}\int_\Omega Q^h(x, \sym\imath(m) + \sym\nabla_h\psi^h_m)\dd x\,,\\
\Kh^0(\Psi) &= \lim_{h\downarrow0}\int_\Omega Q^h(x, \sym\imath(m_0) + \sym\nabla_h\psi^h_0)\dd x\,.
\end{align*}
Using the orthogonality property (\ref{eq:orth_prop_2}) and tricks as in Section \ref{sec:varder}, we calculate:
\begin{align}
\frac{\delta \Kh(m(\varrho,\Psi))}{\delta\varrho}[\phi] \label{eq:dKhdphi}
&= \lim_{h\downarrow0}\int_\Omega \A^h(\sym \imath(m) 
+ \sym\nabla_h\psi_m^h):\imath(\phi e_1)\dd x\,,\\
\frac{\delta \Kh^0(\Psi)}{\delta\Psi}[\Phi] 
&= \lim_{h\downarrow0}\int_\Omega \A^h(\sym \imath(m_0) + \sym\nabla_h\psi_0^h):\sym\imath(\Phi\projyz)\dd x\,,
\label{eq:dKhdPsi1}
\end{align}
for all $\phi\in C_0^\infty(0,L)$ and $\Phi\in C_0^\infty((0,L),\R_{\skw}^{3\times3})$.
On the other hand, from the representation of function $Q_1^0$ as a pointwise quadratic form, we have
\begin{equation}\label{eq:dKhdPsi2}
\frac{\delta \Kh^0(\Psi)}{\delta\Psi}[\Phi] 
= \int_0^L \A_1^0(x_1)\axl \Psi(x_1)\cdot\axl \Phi(x_1)\dd x_1\,.
\end{equation}
Now, if we consider $m(x) = (u' + \frac12((v_2')^2 + (v_3')^2))e_1 + A'\projyz$,
it follows from formulae (\ref{eq:dKhdphi}) and (\ref{eq:phi1}) that
\begin{equation*}
\frac{\delta \Kh(m(a,A'))}{\delta\varrho}[\phi] = 0
\end{equation*}
for all $\phi\in C_0^\infty(0,L)$, where $a(x_1) = u' + \frac12((v_2')^2 + (v_3')^2)$.
In particular, this implies the optimality of the function $a$ for matix function $A'$ in the sense that 
$Q_1^0(\cdot,\axl A') = Q^0(\cdot,a,\axl A')$.
Equating expressions in (\ref{eq:dKhdPsi1}) and (\ref{eq:dKhdPsi2}) for $\Psi = A'$ 
and $\varrho_0 = a$, we obtain the identity
\begin{equation*}
\int_0^L \A_1^0(x_1)\axl A'(x_1)\cdot\axl \Phi(x_1)\dd x_1 = \int_\Omega E:\imath(\Phi\projyz)\dd x\,,
\end{equation*}
for all $\Phi\in C_0^\infty((0,L),\R_{\skw}^{3\times3})$.
From the latter we recognize the following system
\begin{align*}
-(\A_1^0\axl A')_{3}  &= \wtilde E_{11}\,,\\
(\A_1^0\axl A')_2 &= \what E_{11}\,,\\
-(\A_1^0\axl A')_1 &= \what E_{12} - \wtilde E_{13}\,,
\end{align*}
which is a linear second-order system for the limit displacements $v_2$, $v_3$ 
and the limit twist function $w$, and which needs to be accompanied by the following boundary conditions
$v_i(0)=v'_i(L) = 0$ for $i=2,3$, and $w(0) = 0$.
The obtained boundary-value problem represents the homogenized Euler--Lagrange equations
for the von K\'arm\'an rod model. Finally, the scaled displacement $u$ can deduced from the optimality 
property of the function $a$ for the matix function $A'$ and the initial condition $u(0)=0$.

\section{Stochastic Homogenization}\label{sec:sthom}

In this section we will give an explicit cell formula for the quadratic form $Q^0$ 
(limit energy density in expresion \eqref{matt1}) under the assumption of random 
material along the characteristic dimension of rod.
Providing the cell formula for the limit energy in the stochastic 
setting, we will also recover periodic and almost periodic structures.
The methods we are using here are largely based on works \cite{DuGl15}, \cite{HoVe16} and \cite{ZP06}.
Firstly, we will introduce general notion and tools of stochastic 
homogenization, thereafter we will explore the tools needed for thin structures and finally
derive and prove the cell formulae.

\subsection{Stochastic homogenization}

\begin{definition}%[Dynamical System]
A familiy $(T_x)_{x \in \R^n}$ of measurable bijective mappings $T_x \colon \randomspace \to \randomspace$ 
on the probability space $(\randomspace, \randomalgebra, \randommeasure)$ is called a \emph{dynamical system} 
on $\randomspace$ with respect to $\randommeasure$ if:
\begin{enumerate}
\item $T$ is additive, i.e.\ $T_x \circ T_y = T_{x+y}$ for all $x,y \in \R^n$;
\item $T$ is measure- and measurability-preserving, i.e.\ $T_x B$ is measurable and 
$\randommeasure(T_x B) = \randommeasure(B)$ for all $ x\in \R^n$ and $B \in \randomalgebra$;
\item The mapping $\mathcal A \colon \randomspace \times \R^n \to \randomspace$, defined by 
$\mathcal A(\randomelement, x) = T_x(\randomelement)$, is measurable in the pair of 
$\sigma$-algebras $(\randomalgebra \times \mathcal L^n,\randomalgebra)$, 
where $\mathcal L^n$ denotes the family of Lebesgue measurable sets.
\end{enumerate}
\end{definition}
\noindent The key property, which will allow us to derive the cell formula, is ergodicity.
\begin{definition}%[Ergodicity]
A dynamical system $T$ is called {\em ergodic}, if one of the following (equivalent) conditions is fulfilled:
\begin{enumerate}
\item If $f \colon \randomspace \to \randomspace$ is measurable s.t.~$f(\randomelement) = f(T_x \randomelement)$ 
for all $x \in \R^n$ and a.e.~$\randomelement \in \randomspace$, 
then $f$ is $\randommeasure$-a.e.~equal to a constant. 
\item If for some $B \in \randomalgebra$ for all $x \in \R^n$ the set $(T_x B \cup B) \setminus (T_x B \cap B)$ is a null set, 
then $\randommeasure(B) \in  \{ 0, 1\}$.
\end{enumerate}
\end{definition}
\noindent One of the most important consequences of ergodicity is the famous Birkhoff's ergodicity theorem:
\begin{theorem}\label{thm:birkhoff}
Let $T$ be an ergodic, dynamical system and $g \in L^1(\randomspace)$. Then
\begin{equation}\label{eq:ergodicity}
  \lim_{t\to\infty} \frac 1 { t^n \abs A} \int_{tA} g(T_x \trandomelement) \dx = \int_\randomspace g(\randomelement) \drandommeasure
\end{equation}
for almost all $\trandomelement$, for all bounded Borel sets $A \subset \R^n$ with $\abs A > 0$.
\end{theorem}

Let $L^p(\randomspace)$ denotes the set of measurable $p$-integrable functions $b \colon \randomspace \to \R$.
In order to guarantee that spaces $L^p(\randomspace)$ for $p \geq 1$ are separable we assume that the $\sigma$-algebra 
$\randomalgebra$ is countably generated.
The dynamical system allows for more structure on the space $\randomspace$. Denote by $U(x)$ a unitary operator
\[
  U(x):  L^2(\randomspace) \to L^2(\randomspace), \quad U(x) b = b \circ T_x\,.
\]
If for $b \in L^2(\randomspace)$ and $1 \leq k \leq n$ the limit
\[
 \lim_{h \downarrow 0} \frac { b(T_{h\cdot e_k} \randomelement) - b(\randomelement) } h
\]
exists in the $L^2$-sense, then we call it the $k$-th derivative of $b$ and denote it by $D_k b$. The operators 
$D_k$ are infinitesimal generators of maps $T_{x_k}$. Thus, $iD_1, \ldots, iD_n$
are commuiting, self-adjoint, closed and densely defined linear operators on the 
separable Hilbert space $L^2(\randomspace)$.  Let $\mathcal D_k(\randomspace)$ denotes the domain 
of the operator $D_k$, and define the space $W^{1,2}(\randomspace)$ as 
\[
  W^{1,2}(\randomspace) := \mathcal D_1(\randomspace) \cap \ldots \cap \mathcal D_n(\randomspace),
\]
equipped with norm
\[
\norm b_{W^{1,2}(\randomspace)}^2 =\norm {b}^2_{L^2(\randomspace)} + \sum_{k = 1}^n \norm {D_i b}^2_{L^2(\randomspace)}.
\]
We also define the semi-norm
\[
  |b|_{W^{1,2}(\randomspace)}^2 = \sum_{k = 1}^n \norm {D_i b}^2_{L^2(\randomspace)}\,,
\]
and analogously the following Sobolev-type spaces:
\begin{align*}	
  W^{k, 2}(\randomspace) &:= \{ b \in L^2(\randomspace)\, :\, 
  D^{\alpha_1}_1 \ldots D^{\alpha_n}_n b \in L^2(\randomspace),\ \alpha_1 + \ldots + \alpha_n = k \}\,, \\
  W^{\infty, 2}(\randomspace) &:= \bigcap_{k \geq 0} W^{k,2}(\randomspace)\,.
\end{align*}
Furthermore, we define the set of {\em stochastically smooth} functions as
\[
  \stochsmooth(\randomspace) := \{ f \in W^{\infty, 2}(\randomspace): \forall (\alpha_1, \ldots, \alpha_n) 
  \in \N^n_0, \quad D^{\alpha_1}_1 \ldots D^{\alpha_n}_n b \in L^\infty(\randomspace) \}.
\]
The space $\stochsmooth(\randomspace)$ is dense in $L^2(\randomspace)$ (\cite{BMW94}, Lemma~2.1(b)) 
and separable (\cite{BMW94}, Lemma~2.2).
At this point we would like to emphasize, that in the stochastic setting 
we do not have Poincar\'e or Sobolev estimates. Hence, the $L^2$-integrability of 
higher-order derivatives does not yield an $L^\infty$-bound on the derivatives. 
Especially, the space $W^{1,2}(\randomspace)$ is in general incomplete w.r.t.~to the seminorm 
$|\cdot|_{W^{1,2}(\randomspace)}$. Therefore, we introduce its completion denoted 
as $\stochsobolev^{1,2}(\randomspace)$. Differential operators $D_k$ then
extend uniquely as operators $W^{1,2}(\randomspace) \to L^2(\randomspace)$ 
to continuous operators $\stochsobolev^{1,2}(\randomspace) \to L^2(\randomspace)$. 
The $n$-tuple of differential operators $D = (D_1, \ldots, D_n)$ will be called {\em stochastic gradient}.

We say that elements $\trandomelement \in \randomspace$ are {\em typical}, if the identity in 
the Birkhoff's ergodicity theorem \eqref{eq:ergodicity} holds for all $g \in \stochsmooth(\randomspace)$, 
and a trajectory $x \mapsto T_x \trandomelement$ will be called {\em typical}, if $\trandomelement$ is typical. 
Note that separability of $\stochsmooth(\randomspace)$ implies that almost every $\randomelement \in \randomspace$ 
is typical. This enables us to prove the following.
\begin{lemma}
Let $n = 1$. Then for every $b \in L^2(\randomspace)$ with $\int_\randomspace b(\randomelement) \drandommeasure = 0$,
there exists $g\in \stochsobolev^{1,2}(\randomspace)$ such that
\[
  D_1 g = b\,.
\]
\end{lemma}
\begin{remark}
Notice that the zero mean value is necessary, since $\int Dg = 0$ for 
any $g \in \stochsobolev^{1,2}(\randomspace)$, as well as in general $g 
\notin W^{1,2}(\randomspace)$.
\end{remark}
\begin{proof}
By \cite[Proposition A.9.]{DuGl15}, there exists a decomposition 
\[
  L^2( \randomspace ) = F^2_{pot}(\randomspace) \oplus F^2_{sol}(\randomspace) \oplus \R\,,
\]
where 
\begin{align*}	
F^2_{pot}(\randomspace) &:= \mathrm{Cl}_{L^2} \{ D \chi : \chi \in W^{1,2}(\randomspace) \}\,, \\
F^2_{sol}(\randomspace) &:= \mathrm{Cl}_{L^2} \{ D \times \chi : \chi \in W^{1,2}(\randomspace) \}\,.
\end{align*}
For $n = 1$ we have $D \times \chi = 0$ by definition, and the statement follows. \qedhere
\end{proof}

% Two--Scale Convergence
The concept of two-scale convergence was first introduced by Nguetseng in \cite{Ngu89} for periodic problems, while 
Allaire further developed the concept and methods to a versatile tool \cite{All92}. 
For the stochastic setting, the first definition was given in \cite{BMW94}. However, that concept is not well 
suited for our purpose and we will instead use the following (slightly altered) definitions and results given 
in \cite{ZP06}.

\begin{definition}[Weak stochastic two-scale convergence]
Let $(T_{x}\trandomelement)_{x \in \R^n}$ be a typical trajectory and $(v^\eps)$ bounded sequence of functions 
in $L^2(\Omega)$. We say that $(v^\eps)$ {stochastically weakly two-scale converges} to 
$v^{\trandomelement} \in L^2(\Omega \times \randomspace)$ w.r.t.\ $\trandomelement$ and we write
  $v^\eps \wtwoscale v^{\trandomelement}$ if 
\[
  \lim_{\eps \downarrow 0} \int_{\Omega} v^\eps(x) \varphi(x) b(T_{\eps^{-1} x} \trandomelement) \dx 
  = \int_{\randomspace} \int_\Omega v^{\trandomelement}(x, \randomelement) \varphi(x) b(\randomelement) \dx \drandommeasure
\]
for all $\varphi \in C^\infty_0(\Omega)$ and $b \in \stochsmooth(\randomspace)$.
Vector-valued functions are said to stochastically weakly two-scale converge, 
if every component stochastically weakly two-scale converges. 
\end{definition}
\begin{remark}
The difference in this definition to the original one in \cite{ZP06} is the space $\stochsmooth(\randomspace)$ instead of $C^0(\randomspace)$ for the test functions $b$.
This allows us to skip the assumption of a metric on $\randomspace$.
Observe that the limit $v$ may depend on the choice of the typical element, moreover, the sequence $(v^\eps)$ 
may convergence for some typical elements, while not for others. 
From now on we fix a typical $\trandomelement \in \randomspace$ and supress any dependence on it.
\end{remark}
\begin{definition}[Strong stochastic two-scale convergence]
Let $(v^\eps)\subset L^2(\Omega)$ be a weakly stochastic two-scale convergent sequence
with limit $v^0 \in L^2(\Omega \times \randomspace)$. We say that $(v^\eps)$ converges strongly stochastic 
two-scale to $v^0$ if additonally
\[
  \lim_{\eps \downarrow 0} \int_{\Omega} v^\eps (x) u^\eps(x) \dx = \int_{\randomspace}\int_\Omega  v^0(x,\randomelement)u^0(x,\randomelement) \dx \drandommeasure
\]
for every $(u^\eps)\subset L^2(\Omega)$ weakly stochastically two-scale converging to 
$u^0 \in L^2(\Omega \times \randomspace)$. We denote that by $v^\eps \stwoscale v^0$.

\end{definition}

\begin{lemma}[Extension of the test functions]
If $v^\eps \wtwoscale v$, then 
\[
  \lim_{\eps \downarrow 0} \int_{\Omega} v^\eps(x) \varphi(x) b(T_{\eps^{-1}x_1} \trandomelement) \dx 
  = \int_{\randomspace} \int_\Omega v(x, \randomelement) \varphi(x) b(\randomelement) \dx \drandommeasure
\]
holds also for $b \in L^2(\randomspace)$.
\end{lemma}
\begin{lemma}[Compactness]\label{lem:compactness}
Let $(v^\eps)$ be a bounded sequence in $L^2(\Omega)$,
then there exists a subsequence (not relabeled) and $v \in L^2(\Omega \times \randomspace)$ 
such that $v^{\eps} \wtwoscale v$.
\end{lemma}

\begin{lemma}\label{lem:twoscalegradient}
Let $(u^\eps)$ be a bounded sequence in $W^{1,2}(\Omega)$. Then on a subsequence (not relabeled) 
$u^\eps \weakly u^0$ in $W^{1,2}(\Omega)$ and there exists $u^1 \in L^2( \Omega, \stochsobolev^{1,2}(\randomspace))$ 
such that
\begin{align*}	
  u^\eps \wtwoscale u^0 \quad\text{and}\quad
  \nabla u^\eps \wtwoscale \nabla u^0 + D u^1\,.
\end{align*}
\end{lemma}
\noindent The next lemma shows that convex/quadratic functionals are compatible with this concept of 
two-scale convergence. A similar statement with proof can be found in \cite{HoVe16}.
\begin{lemma}[Lower-semicontinuity and continuity of quadratic functionals]\label{lem:lsc}
Let $(u^{\eps})$ be a bounded sequence in $L^2(\Omega, \R^n)$ such that 
\ $u^\eps \wtwoscale u^0 \in L^2(\Omega \times \randomspace, \R^n)$.
 Let $Q : \randomspace \times \R^n \to [0, \infty)$ be a measurable map such that for 
 a.e.~$\randomelement \in \randomspace$, $Q(\randomelement, \cdot)$ is a bounded positive semidefinite quadratic form, 
 i.e.~there exists $\alpha > 0$ such that
\[
  \abs { Q(\randomelement, v)}  \leq \alpha \abs v^2\,,\quad \forall v \in \R^n\,. 
\]
Then 
\[
  \lim_{\eps\downarrow0} \int_\Omega Q \paren*{ T_{\eps^{-1} x_1} \trandomelement, u^\eps(x) } \dx
  \geq \int_\Omega \int_\randomspace Q\paren[\big]{ \randomelement, u^0(\randomelement, x) } \drandommeasure\dx.
\]
If additionally $u^\eps \stwoscale u^0$, then 
\[
  \lim_{\eps\downarrow0} \int_\Omega Q \paren*{ T_{\eps^{-1} x_1} \trandomelement, u^k(x) } \dx
  = \int_\Omega \int_\randomspace Q\paren[\big]{ \randomelement, u^0(\randomelement, x) } \drandommeasure\dx.
\]
\end{lemma}

\subsection{Application in elasticity}
In this subsection we closely follow \cite{Neu10}, where analogous results where derived for the periodic case. 
Since most of the statements can be proved in the same fashion, we will be skipping those.
In the following we work only with one-dimensional dynamical systems $T$, i.e.~$n = 1$. 
We could assume additional microstructure in the cross section (see for instance \cite{MaVe15} 
for the periodic case of bending plate), but for simplicity omit that.

Let $(\eps_h)$ be a sequence of positive numbers, such that $\eps_h \downarrow 0$ for $h \downarrow 0$. 
The random energy density $W^h:\R^3 \times \randomspace \times \R^{3 \times 3} \to [0,+\infty] $ 
is then defined by
\begin{equation}\label{def:randenergy}
  W^h(x, \randomelement, F) = W( T_{\eps^{-1}_h x_1} \randomelement, F)\,,
\end{equation}
where 
\begin{description}
	\item[(S1)]  for a.e.~$\randomelement \in \randomspace$, $W(\randomelement,\cdot)$ is 
				continuous function on $\R^{3\times3}$;
	\item[(S2)] for a.e.~$\randomelement \in \randomspace$, 
				$W(\randomelement,\cdot)\in \mathcal{W}(\alpha,\beta,\varrho,\kappa)$;
	\item[(S3)] there exists a monotone function $r : \R_+ \to (0,+\infty)$ such that $r(\delta)\downarrow 0$ 
	as $\delta\downarrow0$ and 
	\begin{equation}\label{eq:assss}
	\forall G\in\R^{3\times3}\,, 
		\ \forall h>0\ :\ \esssup_{x\in\Omega}|W^h(x,\randomelement,I+G) - Q^h(x,\randomelement,G)| 
		\leq r(|G|)|G|^2\,,
	\end{equation}
	where $Q^h(x,\randomelement,\cdot)$ are quadratic forms defined as in (H4).
\end{description}
 The limiting material properties depend strongly on the relation between $h$ and $\eps_h$, more specifically 
 on $\gamma \in[0,+\infty]$ defined by 
\[
  \gamma := \lim_{ h \downarrow0 } \frac h {\eps_h}\,.
\]
To study the above introduced energies we need Sobolev-type spaces not only
in $\randomspace$, but also on $\randomspace \times \twodomain$. Hence, we define
\[
  W^{1,2}(\randomspace \times \twodomain) 
  := W^{1,2}(\twodomain, L^2(\randomspace)) \cap L^2( \twodomain, W^{1,2}(\randomspace))\,,
\]
equipped with seminorm
\[
  |u|_{W^{1,2}(\randomspace \times \twodomain)}^2 = \norm { D_1 u}^2_{L^2(\randomspace \times \twodomain) } 
  + \norm { \partial_{2} u}^2_{L^2(\randomspace \times \twodomain) }  
  + \norm { \partial_{3} u}^2_{L^2(\randomspace \times \twodomain) }\,.
\]
Similarly as in the purely stochastic Sobolev space, by $\stochsobolev^{1,2}(\randomspace \times \twodomain)$ we denote
the completion of $W^{1,2}(\randomspace \times \twodomain)$ w.r.t.~the seminorm 
$|\cdot|_{W^{1,2}(\randomspace \times \twodomain)}$.
The following statement about stochastic two-scale limit of scaled gradients can be proved as in \cite{HoVe16}.
\begin{lemma}%[Scaled Gradients]
\label{lem:scaled}
Let $(u^h) \subset W^{1,2}(\Omega, \R^3)$ and $ u^0 \in L^2(\Omega, \R^3)$ such that 
$u^h \to u^0$ strongly in $L^2(\Omega, \R^3)$ and let $(\nabla_h u^h)$ be uniformly bounded 
in $L^2(\Omega, \R^{3 \times 3})$. Then $u^0$ depends only  on $x_1$. 
Moreover,
\begin{enumerate}
\item if $\gamma \in \{ 0, \infty \}$, then there exists
\[
  \begin{cases}
    u^1 \in L^2( (0,L), (\stochsobolev^{1,2}( \randomspace ))^3) 
    \text{ and } u^2 \in L^2 ( (0,L) \times \randomspace, W^{1,2}(\twodomain, \R^3) )\,, & \text{  } \gamma = 0\,, \\
    u^1 \in L^2( \Omega, (\stochsobolev^{1,2}( \randomspace ))^3) 
    \text{ and } u^2 \in L^2 ( I , W^{1,2}(\twodomain, \R^3) )\,, & \text{ } \gamma = \infty\,,
  \end{cases}
\]
and
\[
  \nabla_h u^h \wtwoscale (\partial_1 u^0 + D_1 u^1\, |\, \nabla_{x'} u^2 )\,.
\]
\item If $\gamma \in (0, \infty)$, then there is a subsequence (not relabeled) and a 
function $u^1 \in L^2((0,L), \stochsobolev^{1,2}(\randomspace \times \twodomain, \R^3))$ such that
\[
\nabla_h u^h \wtwoscale (\partial_{1} u^0 + D_1 u^1 \,|\, \frac 1 \gamma\nabla_{x'}u^1)\,.
\]
\end{enumerate}
\end{lemma}

\subsection{Cell formula}

\begin{definition}\label{def:cell}
For a.e.~$\randomelement\in\randomspace$ let $Q(\randomelement,\cdot)$ be quadratic a form associated to the
energy density $W(\randomelement,\cdot)$. For every $\varrho\in\R$ and $\Psi\in\R^{3\times3}_{\skw}$, 
define the mapping $Q_\gamma^0:\R\times\R^3\to\R$ by
\begin{align*}	
&Q^0_{\gamma}(\varrho, \axl \Psi )\\
&\quad := \begin{cases}
    \inf\int_{\randomspace}\int_{\twodomain}  Q\paren[\big]
    {\randomelement, \iota(\varrho e_1 + \Psi \projyz + ( D_1 \Psi^1)\projyz) 
    + (D_1 \vartheta^1\, |\, \nabla_{x'} \vartheta^2)} \dx' \drandommeasure\,, & \gamma = 0\,; \\
    \inf\int_{\randomspace}\int_{\twodomain}  
    Q\paren[\big]{\randomelement, \iota(\varrho e_1 +\Psi \projyz) +  
    \paren[\big]{ D_1 \vartheta^1\, |\, \frac 1 \gamma \nabla_{x'} \vartheta^1}} \dx' \drandommeasure\,, 
     & 0 < \gamma < \infty\,; \\
    \inf\int_{\randomspace}\int_{\twodomain}  
    Q\paren[\big]{\randomelement, \iota(\varrho e_1 + \Psi \projyz) +  (D_1 \vartheta^1\, |\, \nabla_{x'} \vartheta^2)} \dx' \drandommeasure\,, 
    & \gamma =  \infty\,,
        \end{cases}
  \end{align*}
where the infimum is taken over all $\Psi^1, \vartheta^1, \vartheta^2$ statisfying: 
$\Psi^1 \in \stochsobolev^{1,2}(\randomspace)^{3\times3}_{\skw}$, 
\[
  \vartheta^1 \in\begin{cases}
                   \stochsobolev^{1,2}(\randomspace)^3\,, & \gamma = 0\,, \\
                   \stochsobolev^{1,2}(\randomspace \times \twodomain)^3\,, & 0 < \gamma < \infty\,, \\
                   L^2( \twodomain, \stochsobolev^{1,2}(\randomspace )^3)\,, & \gamma = \infty\,,
                  \end{cases}\quad \text{ and } \quad
  \vartheta^2 \in\begin{cases}
                   L^2( \randomspace, W^{1,2}(\twodomain, \R^3))\,, & \gamma = 0\, ,\\
                   W^{1,2}(\twodomain, \R^3)\,, & \gamma = \infty\,.
                  \end{cases}
\]
\end{definition}
\begin{proposition}
Let $(W^h)$ be a family of energy densities describing a random material for rods defined by (\ref{def:randenergy}).
Then the limit energy density $Q^0$, defined in (\ref{matt1}), is given by $Q^0_{\gamma}$ 
from Definition \ref{def:cell}.
\end{proposition}

\begin{proof} 
We only prove the result for $0 < \gamma < \infty$. The other two cases are very similar.
Using the previous general homogenization result it suffices to prove that 
for $m=m(\varrho,\Psi)= \varrho e_1 + \Psi \projyz$ it holds
\[
  \lim_{r \downarrow 0} \paren[\Big]{\frac 1 {2r}\Kh(m, x_1^0 + (-r,r) )} = Q^0_{\gamma} ( \varrho, \axl \Psi )\,,
\]
for all  $\varrho\in \R$ and $\Psi\in \R_{\skw}^{3\times3}$, for every Lebesgue point $x_1^0$, where $\Kh$ 
is given by \eqref{2:eq.minseq}. By Lemma~\ref{lemma:key}, for given $m$, there exist sequences of functions 
$(\Psi^{h})\subset H^1((0,L),\R^{3\times3}_{\skw})$ and $(\vartheta^{h})\subset H^1(\Omega,\R^3)$, 
with properties stated there, such that 
\begin{align*}	
    \lim_{r \downarrow 0} &\paren[\Big]{\frac 1 {2r}\Kh(m, x_1^0 + (-r,r) )} = \\
    &\lim_{r \downarrow 0} \frac 1 {2r} \lim_{h \downarrow 0} \int_{(x_1^0 + (-r,r)) 
    \times \twodomain} Q\paren[\Big]{ T_{\eps^{-1} x_1} 
    \trandomelement, \iota(m) + \sym \iota ( (\Psi^{h})' \projyz ) + \sym \nabla_h \vartheta^{h} }\dx\,.
\end{align*}
 Using the lower-semicontinuity of quadratic functionals with respect to the stochastic 
 two-scale convergence we obtain
\begin{align*}	
    \lim_{r \downarrow 0} &\paren[\Big]{\frac 1 {2r}\Kh(m, x_1^0 + (-r,r) )}  \\
    &=\lim_{r \downarrow 0} \frac 1 {2r} \lim_{h \downarrow 0} \int_{(x_1^0 + (-r,r)) 
    \times \twodomain} Q\paren[\Big]{ T_{\eps^{-1} x_1} \trandomelement, 
    \iota(m) + \sym ( (\Psi^{h})' \projyz ) + \sym \nabla_h \vartheta^{h} }\dx \\
    &\geq \lim_{r \downarrow 0} \frac 1 {2r}\inf_{U} \int_{(x_1^0 + (-r,r)) }
    \int_{\randomspace\times \twodomain} Q\paren[\Big]{  \randomelement, \iota(m) +  U }\drandommeasure\dx\,,
\end{align*}
where the infimum is taken over all possible two-scale limits of 
\[
  \sym \iota( (\Psi^{h})' \projyz ) + \sym \nabla_h \vartheta^{h}\, ,
\]
i.e. 
\[
 \left \{  \sym \iota( D_1 \Psi^1 \projyz ) + \sym \frac1\gamma\nabla_{x'} \vartheta^1 \ : 
 \ \Psi^1 \in \stochsobolev^{1,2}(\randomspace)^{3\times3}_{\skw}\,,\, 
 \vartheta^1\in\stochsobolev^{1,2}(\randomspace \times \twodomain)^3 \right\}\,.
\]
Notice that the first term can be absorbed into the second one. To show this we define $\widetilde \vartheta^1$ by
\[
  \widetilde \vartheta^1(\randomelement, x') := 
  \begin{pmatrix}
    \Psi_{12}(\randomelement) x_2 + \Psi_{13}(\randomelement) x_3 \\
    -\frac 1 \gamma \widehat\Psi_{12}(\randomelement)  +\Psi_{23}(\randomelement) x_3 \\
    -\frac 1 \gamma \widehat\Psi_{13}(\randomelement)  -\Psi_{23}(\randomelement) x_2 
  \end{pmatrix},
\]
where $\widehat \cdot$ denotes the primitve of the function. A short calculation reveals that
\[
  \sym (\frac1\gamma \nabla_{x'} \widetilde\vartheta^1) = \sym ( (D_1 \Psi) \projyz)\,.
\]
Therefore, the set of weak stochastic two-scale limits is given by
\[
 \left \{   \frac1\gamma \nabla_{x'}  \vartheta^1 \ : \  
 \vartheta^1\in(\stochsobolev^{1,2}(\randomspace \times \twodomain))^3 \right\}\,.
\]
Hence, we deduce
\begin{align*}	
    \lim_{r \downarrow 0} \paren[\Big]{\frac 1 {2r}\Kh(m, x_1^0 + (-r,r) )}    \geq Q^0_{ \gamma}( \varrho,\axl\Psi )\,.
\end{align*}
For the reverse inequality we fix  $\varrho, \Psi$, and 
let $\vartheta^1 \in (W^{1,2}(\randomspace \times \twodomain))^3$  
be such that
\begin{align*}	
\int_{\randomspace}\int_{\twodomain}  Q\paren[\Big]{\randomelement, 
\iota( \varrho e_1 + \Psi \projyz) +  \paren[\big]{ D_1 \vartheta^1\, |\, \frac 1 \gamma \nabla_{x'} \vartheta^1}} \dx' \drandommeasure 
 \leq \eps +  Q^0_{\gamma}(  \varrho, \axl\Psi )\,.
\end{align*}
Defining
\[
  \vartheta^h(x_1, x') = \frac h \gamma \vartheta^1 ( T_{\eps^{-1}x_1} \trandomelement, x')\,,
\]
we observe
\[
   \sym \paren[\big]{\nabla_h \vartheta^h }  \stwoscale  
   \sym \paren[\Big]{ D_1 \vartheta^1 \,|\, \frac 1 \gamma \nabla_{x'} \vartheta^1}\, .
\]
By continuity of quadratic functions w.r.t.~stochastic two-scale convergence we have
\begin{align*}
 &\eps +  Q^0_{\gamma}(  \varrho, \axl\Psi) \\
&\quad\geq  \int_{\randomspace}\int_{\twodomain}  Q\paren[\Big]{\randomelement, \iota( \varrho e_1 + \Psi \projyz) 
+ \sym \paren[\big]{ D_1 \vartheta^1 \,|\, \frac 1 \gamma \nabla_{x'} \vartheta^1}} \dx' \drandommeasure \\
&\quad=\lim_{r \downarrow 0} \frac 1 {2r}\int_{x_1 + (-r,r)} \lim_{h \downarrow 0}  
\int_{\twodomain}  Q\paren[\Big]{T_{\eps^{-1}x_1} \trandomelement, \iota( \rho e_1 + \Psi \projyz) 
+  \sym \paren[\big]{\nabla_h\vartheta^h }} \dx' \drandommeasure\,,
\end{align*}
which finishes the proof.
\end{proof} 

%%%%%%%%%%%%%%%%%%%%%%%%%%%%%%%%%%%%%%%%%%%%%%%%
%\appendix
\section*{Appendix}

\setcounter{theorem}{0}
\renewcommand{\thesection}{A}

\begin{lemma}\label{lema:fmp1}
Let $p>1$, $\Omega\subset\R^d$ open, bounded set and $(u_k)\subset W^{1,p}(\Omega,\R^m)$ a bounded sequence 
such that $(|\nabla u_k|^p)$ is equi-integrable. Let $(s_k)_k$ be an increasing sequence of positive reals 
such that $s_k\to+\infty$ for
$k\to+\infty$. Then there exists a subsequence still denoted by $(u_k)$ and 
a sequence $(z_k)\subset W^{1,\infty}(\Omega,\R^m)$ satisfying: 
$|z_k \neq u_k| \to 0$ as $k\to+\infty$, 
$(|\nabla z_k|^p)$ is equi-integrable and $\|z_k\|_{W^{1,\infty}} \leq C s_k$ for some 
$C>0$ depending only on dimension $d$.
\end{lemma}

\begin{proof}
The proof is implicitly contained in the proof of Lemma 1.2 (decomposition lemma) from \cite{FMP98},
but we include it here for reasons of completeness.   
As in \cite{FMP98}, the proof is divided into two steps. In the first we assume that $\Omega$ 
is an extension domain, while in the second we remove this restriction generalizing 
the statement for an arbitrary open set.

{\em Step 1.} Let $\Omega\subset \R^d$ be an extension domain, i.e.~an open, bounded set for which there exists an
extension operator $T_\Omega : W^{1,p}(\Omega,\R^m)\to W^{1,p}(\R^d,\R^m)$ satisfying:
\begin{equation*}
T_\Omega u = u\quad\text{on }\Omega\,,\qquad \|T_\Omega u\|_{W^{1,p}}\leq C\|u\|_{W^{1,p}}\,.
\end{equation*}
In the following we identify the sequence $(u_k)\subset W^{1,p}(\Omega,\R^m)$ with its 
extension sequence $(T_\Omega u_k)\subset W^{1,p}(\R^d,\R^m)$. Let us introduce the Hardy-Littlewood maximal
function
\begin{equation}\label{def:HLmf}
M(u)(x):=\sup_{r>0}\frac{1}{|B(x,r)|}\int_{B(x,r)}|u(y)|\dd y\,,
\end{equation}
defined for any Borel measurable function $u:\R^d\to \R^m$. It is known that for $p>1$ and $u\in W^{1,p}(\R^d,\R^m)$,
\begin{equation}\label{eq:Mbound}
\|M (u)\|_{L^p} + \|M(\nabla u)\|_{L^p} \leq C\|u\|_{W^{1,p}}\,.
\end{equation}
According to \cite[Lemma 4.1]{FMP98} (cf.~\cite{EG92}), for every $k\in\N$, there exists 
$z_k\in W^{1,\infty}(\R^d,\R^m)$ such that $u_k = z_k$ on the set $\set S_k := \{M(\nabla u_k)(x) < s_k\}$ and
$\|z_k\|_{W^{1,\infty}}\leq C s_k$, where $C>0$ depends only on $d$. 
Using the argument as in the proof of \cite[Proposition A2.]{FJM02}, 
we obtain an estimate on the Lebesgue measure of the complement of set $\set S_k$,
\begin{equation}\label{app.eq:skc}
|\set S_k^c|\leq \frac{C}{s_k^p}\int_{\{|u_k| + |\nabla u_k| > s_k/2\}} (|u_k| + |\nabla u_k|)^p\dd x \,,
%\leq \frac{C}{s_k}\,,
\quad\text{for all }\,k\in\N\,.
\end{equation}
The strong convergence of $(u_k)$ and the equi-integrability property of $(|\nabla u_k|^p)$
imply that $|\set S_k^c| = |u_k\neq z_k|\to0$ as $k\to\infty$. 
Let $A\subset \R^d$ be a bounded open subset, then
due to the fact that 
$\{u_k = z_k\} = \{u_k = z_k\,, \nabla u_k = \nabla z_k\}$, up to a set of the Lebesgue measure zero,
 we have 
 \begin{equation}
 \int_A|\nabla z_k|^p\dd x = \int_{A\cap \set S_k} |\nabla u_k|^p\dd x 
 + \int_{A\cap \set S_k^c} |\nabla z_k|^p\dd x\,,
 \quad\text{for all }\,k\in\N\,.
 \end{equation}
 Since $(|\nabla u_k|^p)$ is equi-integrable, the first term on the right-hand side can be made 
 arbitrary small for $|A|$ small enough. 
 For the second term, using (\ref{app.eq:skc}), we estimate
 \begin{equation*}
 \int_{\set S_k^c} |\nabla z_k|^p\dd x \leq \|\nabla z_k\|_{L^\infty}^p|\set S_k^c| 
 \leq C\int_{\{|u_k| + |\nabla u_k| > s_k/2\}} (|u_k| + |\nabla u_k|)^p\dd x\,,
 \quad\text{for all }\,k\in\N\,,
 \end{equation*}
 and conclude, as above, that $\lim_{k\to\infty}\int_{\set S_k^c} |\nabla z_k|^p\dd x = 0$.
 Hence, we proved that for every $\eps > 0$ there exists $\delta > 0$ and $k_0\in\N$, such that
 for all open subsets $A\subset\R^d$ with $|A|\leq \delta$ and for all $k\geq k_0$ it holds
 \begin{equation*}
 \int_A|\nabla z_k|^p\dd x \leq \eps\,,
 \end{equation*}
% $$
% |\nabla z_k(x)| = |\nabla u_k(x)| \leq M(\nabla u_k)(x)\,.
% $$ 
% On the other hand, for a.e.~$x\in\set S_k^c$, it holds 
% $$
% |\nabla z_k(x)|\leq Cs_k \leq CM(\nabla u_k)(x)\,.
% $$
% Thus, we conclude that $|\nabla z_k(x)|^p\leq C|M(\nabla u_k)(x)|^p$ for a.e.~$x\in\Omega$, which together with
% the uniform bound (\ref{eq:Mbound}) 
which by definition means the equi-integrability of the sequence $(|\nabla z_k|^p)$.

{\em Step 2.} Let $\Omega$ be an arbitrary open, bounded set. For a given bounded sequence 
$(u_k)\subset W^{1,p}(\Omega,\R^m)$, there exists a subsequence such that
\begin{equation*}
u_k\rightharpoonup u\quad\text{in }W^{1,p}(\Omega,\R^m)\,,\quad u_k\to u\quad\text{in }L^p_{loc}(\Omega,\R^m)\,.
\end{equation*}
Let $(\Omega_l)$ be an increasing sequence of compactly contained subdomains of $\Omega$ satisfying
$|\Omega\backslash\Omega_l|\to0$ as $l\to\infty$, and let $(\zeta_l)\subset C_0^\infty(\Omega,[0,1])$ be a sequence
of cut-off functions such that $\zeta_l(x) = 1$ for  $x\in\Omega_l$. Define $\tilde{u}_k:= u_k-u$, and observe that
\begin{equation*}
\limsup_{l\to\infty}\limsup_{k\to\infty}\|\zeta_l\tilde u_k\|_{L^p} = 0\,
\end{equation*}
and
\begin{align*}
\limsup_{l\to\infty}\limsup_{k\to\infty}\|\nabla(\zeta_l\tilde u_k)\|_{L^p} &= 
\limsup_{l\to\infty}\limsup_{k\to\infty}\|\nabla\zeta_l\otimes\tilde u_k + \zeta_l\nabla\tilde u_k\|_{L^p}\\
&\leq \limsup_{k\to\infty}\|\nabla\tilde u_k\|_{L^p} < \infty\,.
\end{align*}
Then, a standard diagonalization procedure applies (cf.~\cite[Lemma 1.15]{Att84}) and provides a bounded 
sequence $(\zeta_{l(k)}\tilde u_k)\subset W_0^{1,p}(\Omega,\R^m)$, which can be extended by zero to $\R^d$.
Since, $(|\nabla(\zeta_{l(k)}\tilde u_k)|^p)$ is equi-integrable, applying the arguments of Step 1, there
exists a sequence $(\tilde z_k)\subset W^{1,p}(\Omega,\R^m)$ satisfying: 
$|\tilde z_k \neq \zeta_{l(k)}\tilde u_k| \to 0$ 
as $k\to+\infty$, $(|\nabla \tilde z_k|^p)$ is equi-integrable and 
$\|\tilde z_k\|_{W^{1,\infty}} \leq C s_k$ for some $C>0$. 
Since, 
$|\tilde z_k + u \neq u_k| \leq |\tilde z_k \neq \zeta_{l(k)}\tilde u_k| + |\Omega\backslash\Omega_{l(k)}|\to 0$, 
$(|\nabla(\tilde z_k + u)|^p)$ is equi-integrable, 
and $\|\tilde z_k + u\|_{W^{1,\infty}} \leq C s_k$ for some $C>0$, we identify $z_k=\tilde z_k + u$ as the
sought sequence.
\end{proof}

\begin{remark}
If we assume in the previous lemma that $\Omega$ is a Lipschitz domain, as it is the case in our model
of the rod, where $\Omega = (0,L)\times\omega$ and $\omega$ is Lipschitz, then $\Omega$ is also an extension domain
and according to the arguments in Step 1, we can replace the whole sequence $(u_k)$ by its Lipschitz counterpart.
\end{remark}

The following corollary provides the analogous statement to the previous lemma, but with the gradients replaced by
the scaled gradients. A general idea for proving such results can be found in \cite{BrZe07} 
(cf.~also \cite[proof of Lemma 2.17]{MaVe15}), therefore, we omit the proof here.
\begin{corollary}\label{cor:fmp2}
Let $p>1$, $\Omega\subset\R^d$ open, bounded set, 
$(h_k)$ monotonically decreasing zero sequence of positive numbers,
 and $(u_k)\subset W^{1,p}(\Omega,\R^m)$ a bounded sequence 
such that $(\nabla_{h_k}u_k)$ is bounded in $L^{p}(\Omega,\R^m)$ and $(|\nabla_{h_k} u_k|^p)$ is equi-integrable. 
Let $(s_k)_k$ be an increasing sequence of positive reals 
such that $s_k\to+\infty$ for
$k\to+\infty$. Then, there exists a subsequence still denoted by $(u_k)$ and 
a sequence $(z_k)\subset W^{1,\infty}(\Omega,\R^m)$ satisfying: 
$|z_k \neq u_k| \to 0$ as $k\to+\infty$, 
$(|\nabla_{h_k} z_k|^p)$ is equi-integrable, $\|z_k\|_{W^{1,\infty}} \leq C s_k$ 
and $\|\nabla_{h_k}z_k\|_{L^\infty}\leq Cs_k$ for some 
$C>0$ depending only on dimension $d$.
\end{corollary}

\section*{Acknowledgment}
This work has been supported by the Croatian Science
Foundation under Grant agreement No.~9477 (MAMPITCoStruFl).

%%%%%%%%%%%%%%%%%%%%%%%%%%%%%%%%%%%%%%%%%%%%%%%%%%%%%%%%%%%%%

\end{document}